\documentclass[11pt,reqno]{amsart}

\usepackage{amsmath, amsthm, amssymb}
\usepackage{amsfonts}
\usepackage[ansinew]{inputenc}
\usepackage[dvips]{epsfig}
\usepackage{graphicx}
\usepackage[english]{babel}
\usepackage{comment}
\usepackage{graphicx}
\usepackage{hyperref}

\usepackage{cite}
\usepackage{graphicx}
\usepackage{amscd}
\usepackage{xcolor}
\usepackage{bm}
\usepackage{enumerate}

\usepackage{verbatim}
\usepackage{hyperref}
\usepackage{amstext}
\usepackage{latexsym}
\let\oldsqrt\sqrt
\def\sqrt{\mathpalette\DHLhksqrt}
\def\DHLhksqrt#1#2{%
	\setbox0=\hbox{$#1\oldsqrt{#2\,}$}\dimen0=\ht0
	\advance\dimen0-0.2\ht0
	\setbox2=\hbox{\vrule height\ht0 depth -\dimen0}%
	{\box0\lower0.4pt\box2}}

\allowdisplaybreaks

\newcommand{\R}{\mathbb{R}} 
\newcommand{\N}{\mathbb{N}} 

\newcommand{\ov}{\overline}

\renewcommand{\phi}{\varphi}

\newcommand{\cE}{{\mathcal E}}

\newcommand{\cM}{{\mathcal M}}

\newcommand{\M}{\cM}

\theoremstyle{definition}
\newtheorem{defi}{Definition}[section]
\newtheorem{remark}[defi]{Remark}

\theoremstyle{plain} 
\newtheorem{thm}[defi]{Theorem}
\newtheorem{prop}[defi]{Proposition}
\newtheorem{lemma}[defi]{Lemma}

\theoremstyle{definition}

\numberwithin{equation}{section}

\title[On $s$-dependence of weak solution for regional fractional Laplacian]{On the $s$-derivative of weak solutions of the Poisson problem for the regional fractional Laplacian}

\author[Remi Yvant Temgoua]{Remi Yvant Temgoua}  

\address{Goethe-Universit\"{a}t Frankfurt, Institut f\"{u}r Mathematik.
	Robert-Mayer-Str. 10, D-60629 Frankfurt, Germany, and African Institute for Mathematical Sciences in Senegal (AIMS Senegal), KM 2, Route de Joal, B.P. 1418. Mbour, S\'{e}n\'{e}gal.}

\email{temgoua@math.uni-frankfurt.de,~remi.y.temgoua@aims-senegal.org}

\date{\today}

\begin{document}
	\maketitle

	\begin{abstract}
		In this paper, we analyze the $s$-dependence of the solution $u_s$ to the fractional Poisson equation $(-\Delta)^s_{\Omega}u_s=f$ in an open bounded set $\Omega\subset\R^N$. Precisely, we show that the solution map $(0,1)\rightarrow L^2(\Omega)$,~ $s\mapsto u_s$ is continuously differentiable. Moreover, when $f=\lambda_su_s$, we also analyze the one-sided differentiability of the first nontrivial eigenvalue of $(-\Delta)^s_{\Omega}$ regarded as a function of $s\in(0,1)$.
	\end{abstract}
	
	{\footnotesize
		\begin{center}
		\end{center}
	}
	
	\section{Introduction}\label{introduction}
	In a bounded domain $\Omega\subset\R^N~(N\geq2)$ with $C^{1,1}$ boundary, we consider the following nonlocal  Poisson problem 
	\begin{equation}\label{poisson-problem-for-regional-fractional-laplacian}
	(-\Delta)^s_{\Omega}u_s=f\quad\text{in}\quad\Omega,
	\end{equation}
	where, $s\in(0,1)$ and $f\in L^{\infty}(\Omega)$ with $\int_{\Omega}f\ dx=0$. Here,  $(-\Delta)^s_{\Omega}$ stands for the regional fractional Laplacian define for all $u\in C^2(\ov\Omega)$ by
	\begin{equation*}
	(-\Delta)^s_{\Omega}u(x)=C_{N,s}P.V.\int_{\Omega}\frac{u(x)-u(y)}{|x-y|^{N+2s}}\ dy,\qquad x\in\Omega,
	\end{equation*}
	where $P.V.$ is a commonly used abbreviation for ''in the principal value sense'' and the normalized constant $C_{N,s}$ is defined by
	\begin{equation}\label{bound-of-constant-c-n-s}
	C_{N,s}=s(1-s)\pi^{-\frac{N}{2}}2^{2s}\frac{\Gamma(\frac{N+2s}{2})}{\Gamma(2-s)}\in\Big(0,4\Gamma\Big(\frac{N}{2}+1\Big)\Big],
	\end{equation}
	$\Gamma$ being the usual Gamma function. The bounds in \eqref{bound-of-constant-c-n-s} can be found in \cite[page $8$]{terracini2018s}. As known, $(-\Delta)^s_{\Omega}$ represents the infinitesimal generator of the so-called \textit{censored stable L\'{e}vy process}, that is, a stable process in which the jumps between $\Omega$ and its complement are forbidden (see e.g.  \cite{andreu2010nonlocal,bogdan2003censored,chen2018dirichlet,guan2006integration,guan2006reflected,warma2015fractional} and the references therein). 
	
	The study of the $s$-dependence of the solution of the Poisson problem involving nonlocal operators has recently received quite some interest. This kind of study allows a well understanding of the asymptotic behavior of the solution at $s\in(0,1)$. For instance, in \cite{biccari2018poisson}, the authors analyzed the limit behavior as $s\rightarrow1^-$ of the solution to the fractional Poisson equation $(-\Delta)^su_s=f_s$, $x\in\Omega$ with homogeneous Dirichlet boundary conditions $u_s\equiv0$, $x\in\R^N\setminus\Omega$ and provided continuity in a weak setting. We also refer to \cite{deblassie2005alpha} where the equicontinuity of eigenfunctions and eigenvalues of $(-\Delta)^s$ in $\Omega$ was studied for $s$ belonging in a compact subset of $(0,1)$. Small order asymptotics of eigenvalue problems for the operators $(-\Delta)^s$ and $(-\Delta)^s_{\Omega}$ has been studied recently in \cite{chen2019dirichlet,feulefack2022small,temgoua2021eigenvalue}. We notice that this type of optimization in the small order-dependent
	appears in population dynamics \cite{pellacci2018best}, in optimal control \cite{antil2018optimization,sprekels2017new}, in fractional harmonic maps \cite{antil2021approximation}, and in image processing \cite{antil2017spectral}. 
	
	Recently, Burkovska and Gunzburger \cite[section 5]{burkovska2020affine} studied the $s$-regularity of solutions to Dirichlet problems driven by the fractional peridynamic operator. Precisely, they proved that the solution map $(0,1)\to L^2(\Omega),~s\mapsto u_s$ is of class $C^{\infty}$. Their argument is based on differentiating the corresponding bilinear form of the problem with respect to $s$. In \cite{jarohs2020new}, Jarohs, Salda$\tilde{\text{n}}$a, and Weth established the $C^1$-regularity of the map $(0,1)\rightarrow L^{\infty}(\Omega)$, $s\mapsto u_s$, where $u_s$ is given as the unique weak solution to the fractional Poisson problem $(-\Delta)^su_s=f$, $x\in\Omega$ with the homogeneous Dirichlet boundary data $u_s\equiv0$, $x\in\R^N\setminus\Omega$, where $\Omega$ is a bounded domain with $C^2$ boundary and $f\in C^{\alpha}(\ov\Omega)$ for some $\alpha>0$. The main advantage in their analysis relies on the representation formula of the solution $u_s$ by mean of Green function which allows obtaining several important and powerful estimates. 
	
	The purpose of the present paper is to analyze the $C^1$-regularity of the map $(0,1)\rightarrow L^2(\Omega)$, $s\mapsto u_s$, where $u_s$ is the unique solution of the fractional Poisson problem \eqref{poisson-problem-for-regional-fractional-laplacian}. The major difficulty in our development stems from the fact that, contrary to \cite{jarohs2020new}, we do not have an explicit representation of the solution $u_s$ in terms of Green function for every $s\in(0,1)$. We note that for $s\in(0,\frac{1}{2}]$, the fractional Poisson problem $(-\Delta)^s_{\Omega}u_s=f$, $x\in\Omega$ with Dirichlet boundary conditions $u_s\equiv g$, $x\in\partial\Omega$ remains ill-posed since the space $C^{\infty}_c(\Omega)$ is dense in the fractional Sobolev space $H^s(\Omega)$. We refer to \cite{bogdan2003censored} and the references therein for the probabilistic interpretation and to Chen and Wei \cite{chen2020non-existence} for recent results from a purely analytic point of view. 
	
	Throughout the paper, we consider \eqref{poisson-problem-for-regional-fractional-laplacian} as a free Poisson problem, so without Dirichlet boundary condition.
	
	In the following, we present the main results of the present paper. The first main result deals with the differentiability of the solution map $(0,1)\to L^2(\Omega),~ s\mapsto u_s$. It reads as follows. 
	
	\begin{thm}\label{differentiability-of-u_s}
		Let $f\in L^{\infty}(\Omega)$ with $\int_{\Omega}f\ dx=0$ and let $u_s$ be the unique weak solution of \eqref{poisson-problem-for-regional-fractional-laplacian} (see Section \ref{preliminary} below for the definition of weak solution). Then the map
		\begin{equation*}
		(0,1)\rightarrow L^2(\Omega),\quad s\mapsto u_s
		\end{equation*}
		is of class $C^1$ and $w_s:=\partial_su_s$ uniquely solves in the weak sense the equation
		\begin{equation*}
		(-\Delta)^s_{\Omega}w_s=M^s_{\Omega}u_s\quad\text{in}\quad\Omega,
		\end{equation*}
		where for every $x\in\Omega$,
		\begin{equation*}
		M^s_{\Omega}u(x)=-\frac{\partial_sC_{N,s}}{C_{N,s}}f(x)+2C_{N,s}P.V.\int_{\Omega}\frac{(u(x)-u(y))}{|x-y|^{N+2s}}\log|x-y|\ dy.
		\end{equation*}
	\end{thm}
	
	We now consider the problem \eqref{poisson-problem-for-regional-fractional-laplacian} with $f\equiv\lambda_su_s$ i.e., the eigenvalue problem
	\begin{equation*}
	(-\Delta)^s_{\Omega}u_s=\lambda_su_s\quad\text{in}~~\Omega.
	\end{equation*}
	
	Our second main result is concerned with the one-sided differentiability of the map $s\mapsto \lambda_{1,s}$ where $\lambda_{1,s}$ is the first nontrivial eigenvalue of $(-\Delta)^s_{\Omega}$ in $\Omega$ (see Section \ref{eigenvalue-problem-case} below for the definition of $\lambda_{1,s}$). It reads as follows.
	
	\begin{thm}\label{differentiability-of-eigenvalues-0}
		Regarded as function of $s$, $\lambda_{1,s}$ is right differentiable on $(0,1)$ and
		\begin{equation}\label{C-1-regularity-of-eigenvalues-0}
		\partial^+_s\lambda_{1,s}:=\lim\limits_{\sigma\rightarrow0^+}\frac{\lambda_{1,s+\sigma}-\lambda_{1,s}}{\sigma}=\inf\{J_s(u):u\in \M_s\}
		\end{equation}
		where
		\begin{equation*}
		J_s(u)=\frac{\partial_sC_{N,s}}{C_{N,s}}\lambda_{1,s}-C_{N,s}\int_{\Omega}\int_{\Omega}\frac{(u(x)-u(y))^2}{|x-y|^{N+2s}}\log|x-y|\ dxdy
		\end{equation*}
		and $\M_s$ the set of $L^2$-normalized eigenfunctions of $(-\Delta)^s_{\Omega}$ corresponding to $\lambda_{1,s}$. 
		Moreover, the infimum in \eqref{C-1-regularity-of-eigenvalues-0} is attained. 
	\end{thm}
	
	The stategy of the proof of Theorems \ref{differentiability-of-u_s} and \ref{differentiability-of-eigenvalues-0} is based on differentiating the quadratic form $\cE_s$ (see \eqref{quadratic-form}) with respect to $s$. To this end, higher Sobolev regularity of any weak solution $u_s$ of \eqref{poisson-problem-for-regional-fractional-laplacian} of the form $s+\epsilon$ is needed, see Proposition \ref{higher-sobolev-regularity} below. The latter is obtained by exploiting boundary regularity result by Fall \cite{fall2020regional}. Let us also mention that Proposition \ref{higher-sobolev-regularity} plays a crucial role in getting uniform $H^s(\Omega)$-estimates of $u_{s+\sigma}$ and the  difference quotient $\frac{u_{s+\sigma}-u_s}{\sigma}$ with respect to $\sigma$. 
	
	The paper is organized as follows. In Section \ref{preliminary} we present some preliminaries that will be useful throughout this article. In Section \ref{differentiability-for-general-value-s}, we prove Theorem \ref{differentiability-of-u_s}. Finally, Section \ref{eigenvalue-problem-case} is devoted to the proof of Theorem \ref{differentiability-of-eigenvalues-0}.  \\
	
	\textbf{Acknowledgements:} This work is supported by DAAD and BMBF (Germany) within project 57385104. The author is grateful to Tobias Weth, Mouhamed Moustapha Fall, and Sven Jarohs for helpful discussions.

	\section{Preliminary and functional setting}\label{preliminary}
	In this section, we introduce some preliminary properties that will be useful in this work. First of all, throughout the end of the paper, $d_{A}:=\sup\{|x-y|:x,y\in A\}$ is the diameter of $A\subset\R^N$ and $B_r(x)$ denotes the open ball centered at $x$ with radius $r$. We also denote by $|A|$ the $N$-dimensional Lebesgue measure of every set $A\subset\R^N$. \\
	Now, for all $s\in(0,1)$ the usual fractional Sobolev space $H^s(\Omega)$ is defined by
	\begin{align*}
	H^s(\Omega)=\{u\in L^2(\Omega):|u|^2_{H^s(\Omega)}<\infty\},
	\end{align*}
	where 
	\begin{align*}
	|u|_{H^s(\Omega)}:=\Big(\int_{\Omega}\int_{\Omega}\frac{|u(x)-u(y)|^2}{|x-y|^{N+2s}}\ dxdy\Big)^{1/2},
	\end{align*}
	is the so-called Gagliardo seminorm of $u$. Moreover, $H^s(\Omega)$ is a Hilbert space endowed with the norm 
	\begin{align*}
	\|u\|_{H^s(\Omega)}:=(\|u\|^2_{L^2(\Omega)}+|u|^2_{H^s(\Omega)})^{1/2}.
	\end{align*}
	Also, we notice the following continous embedding $H^t(\Omega)\hookrightarrow H^s(\Omega)$ for $t>s$, see e.g. \cite[Proposition 2.1]{di2012hitchhiker's}. 
	It is also useful to recall that the space  $C^{\infty}(\overline{\Omega})$ is dense in $H^s(\Omega)$ (see \cite[Corollary 2.71]{demengel2012functional}). We recall that $C^{\infty}(\ov\Omega)$ denotes the restriction of all $C^{\infty}(\R^N)$ functions on $\ov\Omega$. Moreover, we define the space $\mathbb{X}^s(\Omega)$ consists of functions in $H^s(\Omega)$ orthogonal to constants i.e., 
	\begin{align*}
	\mathbb{X}^s(\Omega):=\Big\{u\in H^s(\Omega):\int_{\Omega}u\ dx=0\Big\}.
	\end{align*}
	Clearly, $\mathbb{X}^s(\Omega)$ is a Hilbert space (with the norm $\|\cdot\|_{\mathbb{X}^s(\Omega)}:=|\cdot|_{H^s(\Omega)}$ equivalent to the usual one in $H^s(\Omega)$) for which every function $u\in\mathbb{X}^s(\Omega)$ satisfies the following fractional Poincar\'{e} inequality
	\begin{equation}\label{Poincare-inequality}
	\|u\|^2_{L^2(\Omega)}\leq\gamma_{N,s,\Omega}|u|^2_{H^s(\Omega)}\quad\text{with}~~\gamma_{N,s,\Omega}=|\Omega|^{-1}d^{N+2s}_{\Omega}.
	\end{equation}
	We notice also that the space of functions $\phi\in C^{\infty}(\overline{\Omega})$ with $\int_{\Omega}\phi\ dx=0$ is dense in $\mathbb{X}^s(\Omega)$. For simplicity, we set $C^{\infty}_0(\overline{\Omega}):=\{u\in C^{\infty}(\overline{\Omega}):\int_{\Omega}\phi\ dx=0\}$. 
	The inner product and the norm in $L^2(\Omega)$ will be denoted by $\langle\cdot,\cdot\rangle_{L^2(\Omega)}$ and $\|\cdot\|_{L^2(\Omega)}$ respectively.
	
	Now, let $\cE_s$ be the quadratic form define on $H^s(\Omega)$ by
	\begin{equation}\label{quadratic-form}
	(u,v)\mapsto\cE_s(u,v)=\frac{C_{N,s}}{2}\int_{\Omega}\int_{\Omega}\frac{(u(x)-u(y))(v(x)-v(y))}{|x-y|^{N+2s}}\ dxdy.
	\end{equation}
	We have the following.
	\begin{defi}
		Let $f\in L^{\infty}(\Omega)$ with $\int_{\Omega}f\ dx=0$. We say that $u_s\in\mathbb{X}^s(\Omega)$ is a weak solution of \eqref{poisson-problem-for-regional-fractional-laplacian} if
		\begin{equation}\label{weak-formulation-direct-method}
		\cE_s(u_s,\phi)=\int_{\Omega}f\phi\ dx,\quad\forall\phi\in\mathbb{X}^s(\Omega).
		\end{equation}
	\end{defi}
	The existence and uniqueness of weak solution of the Poisson problem  \eqref{poisson-problem-for-regional-fractional-laplacian} in $\mathbb{X}^s(\Omega)$ is guaranteed by Riesz representation theorem. 
	
	Let $u_s\in\mathbb{X}^s(\Omega)$ be the unique weak solution of \eqref{poisson-problem-for-regional-fractional-laplacian}. Then, thanks to \cite{temgoua2021eigenvalue}, there exists a constant $c_1>0$ independent of $s$ such that
	\begin{equation}\label{uniformly-boundedness-of-weak-solution}
	\|u_s\|_{L^{\infty}(\Omega)}\leq c_1.
	\end{equation}
	Very recently, Fall \cite{fall2020regional} established boundary regularity for any weak solution to  \eqref{poisson-problem-for-regional-fractional-laplacian}. Presicely, among other results, he proved that $u_s\in C^{\beta}(\overline{\Omega})$ and
	\begin{equation}\label{boundary-regularity-for-regional}
	\|u_s\|_{C^{\beta}(\overline{\Omega})}\leq C(\|u_s\|_{L^2(\Omega)}+\|f\|_{L^p(\Omega)})
	\end{equation}
	with 
	\begin{align*}
	\beta:=2s-\frac{N}{p}
	\end{align*}
	for every $s\in(0,1)$ and $p>\frac{N}{2s}$. Moreover if $s\in(\frac{1}{2},1)$ so that $\beta=2s-\frac{N}{p}>1$, he also obtained boundary H\"{o}lder regularity for the gradient of the form $\nabla u_s\in C^{\beta-1}(\overline{\Omega})$ with
	\begin{align}\label{boundary-regularity-for-the-gradient}
	\|\nabla u_s\|_{C^{\beta-1}(\overline{\Omega})}\leq C(\|u_s\|_{L^2(\Omega)}+\|f\|_{L^p(\Omega)}).
	\end{align}
	From now on and without loss of generality, we fix $p$ such that $p>\frac{N}{s}$. Moreover, the constant $C$ appearing in \eqref{boundary-regularity-for-regional} is continuous at $s\in[s_0,1)$ for some $s_0\in(0,1)$ see \cite[Theorem 1.1]{fall2020regional}. The same conclusion holds for \eqref{boundary-regularity-for-the-gradient} provided that $s\in[s_0,1)$ for some $s_0\in (\frac{1}{2},1)$ see \cite[Remark 1.4]{fall2020regional}. Hence, by taking into account \eqref{uniformly-boundedness-of-weak-solution} we obtain from \eqref{boundary-regularity-for-regional} and \eqref{boundary-regularity-for-the-gradient} uniform bound with respect to $s$ on $C^{\beta}$ and $C^{\beta-1}$ norm of $u_s$ and $\nabla u_s$ respectively as follows
	\begin{equation}\label{boundary-uniform-bound}
	\|u_s\|_{C^{\beta}(\overline{\Omega})}\leq c_2\quad\quad\text{and}\quad\quad\|\nabla u_s\|_{C^{\beta-1}(\overline{\Omega})}\leq c_3.
	\end{equation} 
	
	As a direct advantage of the above boundary regularity, we derive in the next proposition, higher Sobolev regularity for solution of \eqref{poisson-problem-for-regional-fractional-laplacian}.
	\begin{prop}\label{higher-sobolev-regularity}
		Let $f\in L^{\infty}(\Omega)$ with $\int_{\Omega}f\ dx=0$ and let $u_{s+\sigma}\in H^{s+\sigma}(\Omega)\cap L^{\infty}(\Omega)$ be the unique weak solution of problem \eqref{poisson-problem-for-regional-fractional-laplacian} with $s$ replaced by $s+\sigma$. Then $u_{s+\sigma}\in H^{s+\epsilon}(\Omega)$ for some $\epsilon>0$ and
		\begin{equation}\label{b1}
		\|u_{s+\sigma}\|_{H^{s+\epsilon}(\Omega)}\leq K\quad\quad\text{for all}~\sigma\in(-s_0,s_0)
		\end{equation}
		for some $s_0>0$.
	\end{prop}
	
	\begin{proof}
		Let $u_{s+\sigma}\in H^{s+\sigma}(\Omega)\cap L^{\infty}(\Omega)$ be the unique weak solution of \eqref{poisson-problem-for-regional-fractional-laplacian} with $s$ replaced by $s+\sigma$ for all $\sigma\in(-s_0,s_0)$ for some $s_0>0$. Then, \\
		
		$(i)$ If $2s\leq1$ then from \eqref{boundary-regularity-for-regional} we have that
		\begin{align*}
		\int_{\Omega}\int_{\Omega}\frac{|u_{s+\sigma}(x)-u_{s+\sigma}(y)|^2}{|x-y|^{N+2(s+\epsilon)}}\ dxdy
		&\leq\|u_s\|^2_{C^{2(s+\sigma)-\frac{N}{p}}(\overline{\Omega})}\int_{\Omega}\int_{\Omega}|x-y|^{4(s+\sigma)-\frac{2N}{p}-N-2(s+\epsilon)}\ dxdy\\
		&\leq d^{4\sigma}_{\Omega}\|u_s\|^2_{C^{2s-\frac{N}{p}}(\overline{\Omega})}\int_{\Omega}\int_{B_{d_{\Omega}}(x)}|x-y|^{2s-\frac{2N}{p}-N-2\epsilon}\ dydx\\
		&\leq\frac{\max\{d^{-4s_0}_{\Omega},d^{4s_0}_{\Omega}\}|S^{N-1}||\Omega|\|u_s\|^2_{C^{2s-\frac{N}{p}}(\overline{\Omega})}}{2(s-\epsilon-\frac{N}{p})}d^{2(s-\epsilon-\frac{N}{p})}_{\Omega}<\infty
		\end{align*}
		provided that $0<\epsilon<s-\frac{N}{p}$. This shows that $u_{s+\sigma}\in H^{s+\epsilon}(\Omega)$ with uniform bound in $\sigma\in (-s_0,s_0)$ provided that $0<\epsilon<s-\frac{N}{p}$. \\
		
		$(ii)$ If $2s>1$ then from \eqref{boundary-regularity-for-the-gradient} we have that $\nabla u_{s+\sigma}\in L^{\infty}(\overline{\Omega})\subset L^{\infty}(\Omega)$ and therefore $\nabla u_{s+\sigma}\in L^2(\Omega)$. Since also $u_{s+\sigma}\in L^2(\Omega)$, we deduce that $u_{s+\sigma}\in H^1(\Omega)$. Hence,
		\begin{align}
		&u_{s+\sigma}\in H^s(\Omega)\cap H^1(\Omega)\quad\quad\quad\quad\text{for all}~~\sigma\in[0,s_0)\label{1}\\
		&u_{s+\sigma}\in H^{\alpha(s)}(\Omega)\cap H^1(\Omega)\quad\quad\quad\text{for all}~~\sigma\in(-s_0,0]\label{2}
		\end{align}
		for some $\alpha(s)<<s$ depending only on $s$. Applying the  well-known Gagliardo-Nirenberg interpolation inequality (see e.g. \cite[Theorem 1]{brezis2018gagliardo}), we find that $u_{s+\sigma}\in H^r(\Omega)$ with 
		\begin{itemize}
			\item [$(a)$] $r=\theta s+(1-\theta)\cdot1~~\text{for all}~\theta\in (0,1)$ in the case \eqref{1}, and 
			\begin{align}\label{interpolation-inequality}
			\|u_{s+\sigma}\|_{H^r(\Omega)}\leq C(\theta,s,\Omega)\|u_{s+\sigma}\|^{\theta}_{H^s(\Omega)}\|u_{s+\sigma}\|^{1-\theta}_{H^1(\Omega)}.
			\end{align}
			\item[$(b)$] $r=\theta \alpha(s)+(1-\theta)\cdot1~~\text{for all}~\theta\in (0,1)$ in the case \eqref{2}, and
			\begin{align}\label{interpolation-inequality'}
			\|u_{s+\sigma}\|_{H^r(\Omega)}\leq C(\theta,s,\Omega)\|u_{s+\sigma}\|^{\theta}_{H^{\alpha(s)}(\Omega)}\|u_{s+\sigma}\|^{1-\theta}_{H^1(\Omega)}.
			\end{align}
		\end{itemize}
		Let us focus on the situation $(a)$. By choosing in particular $\theta=\frac{1}{2}$, then $r=\frac{s}{2}+\frac{1}{2}=s+\frac{1-s}{2}$ and we have that $u_{s+\sigma}\in H^{s+\frac{1-s}{2}}(\Omega)$. From this, we conclude that $u_{s+\sigma}\in H^{s+\epsilon}(\Omega)$ for every $0<\epsilon<\frac{1-s}{2}$. To complete the proof, it remains to show that the RHS of \eqref{interpolation-inequality} is uniform for $\sigma$ sufficiently small.
		
		From \eqref{uniformly-boundedness-of-weak-solution} and \eqref{boundary-uniform-bound} we have that
		\begin{equation}\label{b2}
		\|u_{s+\sigma}\|_{H^1(\Omega)}\leq C_1\quad\quad\text{for all}~~\sigma~~\text{sufficiently small}.
		\end{equation}
		On the other hand, since $s<s+\sigma$, then from \cite[Proposition 2.1]{di2012hitchhiker's} there exists $c>0$ depending only on $s$ and $N$ such that
		\begin{equation}\label{b3}
		|u_{s+\sigma}|_{H^{s}(\Omega)}\leq c|u_{s+\sigma}|_{H^{s+\sigma}(\Omega)}.
		\end{equation}
		Using now $u_{s+\sigma}$ as a test function in \eqref{poisson-problem-for-regional-fractional-laplacian} with $s$ replace by $s+\sigma$ and integrating over $\Omega$, one has
		\begin{align}\label{b4}
		|u_{s+\sigma}|^2_{H^{s+\sigma}(\Omega)}=\frac{2}{C_{N,s+\sigma}}\int_{\Omega}fu_{s+\sigma}\ dx\leq\frac{2|\Omega|}{C_{N,s+\sigma}}\|f\|_{L^{\infty}(\Omega)}\|u_{s+\sigma}\|_{L^{\infty}(\Omega)}\leq c\quad\text{as}~~\sigma\to0^+,
		\end{align}
		thanks to \eqref{uniformly-boundedness-of-weak-solution} and the continuity of the map $s\mapsto C_{N,s}$. This, together with \eqref{b3} yield
		\begin{equation}
		\|u_{s+\sigma}\|_{H^s(\Omega)}\leq C_2\quad\quad\text{for}~~\sigma~~\text{sufficiently small}. 
		\end{equation}
		From this, one gets $u_{s+\sigma}\in H^{s+\epsilon}(\Omega)$ with uniform bound in $\sigma\in[0,s_0)$.
		
		In situation $(b)$, a similar argument as above yields $u_{s+\sigma}\in H^{s+\tilde{\epsilon}}(\Omega)$ for some $\tilde{\epsilon}>0$ depending on $s$.  
		
		Now, by combining $(i)$ and $(ii)$, we conclude the proof. 
	\end{proof}
	This higher Sobolev regularity will be of a capital interest in the rest of the paper.\\
	Next, we recall the following decay estimate regarding the logarithmic function. For all $r,\epsilon_0>0$, there holds that
	\begin{equation}\label{logarithmic-decays}
	|\log|z||\leq\frac{1}{e\epsilon_0}|z|^{-\epsilon_0}~~\text{if}~~ |z|\leq r~~~~\text{and}~~~~|\log|z||\leq\frac{1}{e\epsilon_0}|z|^{\epsilon_0}~~ \text{if}~~|z|\geq r.
	\end{equation}
	We end this section by recalling the following.
	\begin{prop}(\cite[Lemma 6.6]{jarohs2020new})\label{diff-of-the-curve}
		Let $I\subset\R$ be an open interval, $E$ be a Banach space and $\gamma:I\to E$ be a curve with the following properties
		\begin{itemize}
			\item [$(i)$] $\gamma$ is continuous.
			\item [$(ii)$] $\partial_s^+\gamma(s):=\lim\limits_{\sigma\to0^+}\frac{\gamma(s+\sigma)-\gamma(s)}{\sigma}$ exists in $E$ for all $s\in I$.
			\item [$(iii)$] The map $I\to E,~s\mapsto\partial_s^+\gamma(s)$ is continuous.
		\end{itemize}
		Then $\gamma$ is continuously differentiable with $\partial_s\gamma=\partial_s^+\gamma$.
	\end{prop}

	\section{Differentiability of the solution map in $(0,1)$}\label{differentiability-for-general-value-s}
	
	In this section, we are concerned with the regularity of the map $(0,1)\rightarrow L^2(\Omega), s\mapsto u_s$, with being $u_s$ the unique weak solution of \eqref{poisson-problem-for-regional-fractional-laplacian}. In order to obtain the regularity of the solution $u_s$, regarded as function of $s$, our strategy consist to bound uniformly the difference quotient $\frac{u_{s+\sigma}-u_s}{\sigma}$ in the Hilbert space $H^s(\Omega)$ with respect to $\sigma$, after what, due to compactness, we therefore reach our goal.
	
	We restate Theorem \ref{differentiability-of-u_s} from the introduction for the reader's convenience.
	\begin{thm}
		Let $f\in L^{\infty}(\Omega)$ with $\int_{\Omega}f\ dx=0$ and let $u_s\in\mathbb{X}^s(\Omega)$ be the unique weak solution of \eqref{poisson-problem-for-regional-fractional-laplacian}. Then the map
		\begin{equation*}
		(0,1)\rightarrow L^2(\Omega),\quad s\mapsto u_s
		\end{equation*}
		is of class $C^1$ and $w_s:=\partial_su_s$ uniquely solves in the weak sense the equation
		\begin{equation}\label{s-derivative-equation}
		(-\Delta)^s_{\Omega}w_s=M^s_{\Omega}u_s\quad\text{in}\quad\Omega,
		\end{equation}
		where for every $x\in\Omega$,
		\begin{equation*}
		M^s_{\Omega}u(x)=-\frac{\partial_sC_{N,s}}{C_{N,s}}f(x)+2C_{N,s}P.V.\int_{\Omega}\frac{(u(x)-u(y))}{|x-y|^{N+2s}}\log|x-y|\ dy.
		\end{equation*}
	\end{thm}
	\begin{proof}
		The proof is devided into three steps.\\
		
		\textbf{Step 1.} We prove that the solution map $(0,1)\to L^2(\Omega), s\mapsto u_s$ is continuous.
		
		Fix $s_0\in(0,1)$. Let $\delta\in(0,s_0)$ and $(s_n)_n\subset(s_0-\delta,1)$ be a sequence such that $s_n\to s_0$. We want to show that
		\begin{equation}\label{z1}
		u_{s_n}\to u_{s_0}\quad\text{in}~~L^2(\Omega)\quad\text{as}\quad\quad n\to\infty.
		\end{equation}
		
		Put $s':=\inf_{n\in\N}s_n>s_0-\delta$. Then,
		\begin{align}\label{z2}
		\nonumber|u_{s_n}|^2_{H^{s'}(\Omega)}&=\int_{\Omega}\int_{\Omega}\frac{(u_{s_n}(x)-u_{s_n}(y))^2}{|x-y|^{N+2s'}}\ dxdy\\
		\nonumber&=\int_{\Omega}\int_{\Omega}\frac{(u_{s_n}(x)-u_{s_n}(y))^2}{|x-y|^{N+2s_n}}|x-y|^{2(s_n-s')}\ dxdy\\
		&\leq d^{2(s_n-s')}_{\Omega}|u_{s_n}|^2_{H^{s_n}(\Omega)}\leq c|u_{s_n}|^2_{H^{s_n}(\Omega)}.
		\end{align}
		Now, since $u_{s_n}$ is the unique weak solution to $(-\Delta)^{s_n}_{\Omega}u_{s_n}=f$ in $\Omega$, then it follows that (we use $u_{s_n}$ as a test function)
		\begin{align}\label{z3}
		|u_{s_n}|^2_{H^{s_n}(\Omega)}=\frac{2}{C_{N,s_n}}\int_{\Omega}fu_{s_n}\ dx\leq\frac{2}{C_{N,s_n}}\|f\|_{L^2(\Omega)}\|u_{s_n}\|_{L^2(\Omega)}\leq \frac{2c_{s'}}{C_{N,s_n}}\|f\|_{L^2(\Omega)}|u_{s_n}|_{H^{s'}(\Omega)}
		\end{align}
		thanks to fractional Poincar\'{e} inequality \eqref{Poincare-inequality}. Using that $s\mapsto C_{N,s}$ is continuous, then it follows from \eqref{z3} and \eqref{z2} that
		\begin{equation}\label{z4}
		|u_{s_n}|_{H^{s'}(\Omega)}\leq c\quad\quad\text{as}\quad n\to\infty.
		\end{equation}
		This means that $(u_{s_n})_n$ is uniformly bounded in $H^{s'}(\Omega)$. Then there is $u_{*}\in H^{s'}(\Omega)$ such that after passing to a subsequence, 
		\begin{equation}\label{z5}
		\begin{aligned}
		&u_{s_n}\rightharpoonup u_{*}\quad\text{weakly in}~~H^{s'}(\Omega),\\
		&u_{s_n}\rightarrow u_{*}\quad\text{strongly in}~~L^2(\Omega),\\
		&u_{s_n}\rightarrow u_{*}\quad\text{a.e. in}~~\Omega. 
		\end{aligned}
		\end{equation}
		In particular, $\int_{\Omega}u_{*}\ dx=0$. We wish now to show that $u_{s_0}\equiv u_{*}$. To this end, we first prove that $u_{*}\in H^{s_0}(\Omega)$.
		
		By Fatou's Lemma, we have 
		\begin{align}\label{z6}
		\nonumber|u_{*}|^2_{H^{s_0}(\Omega)}&=\int_{\Omega}\int_{\Omega}\frac{(u_{*}(x)-u_{*}(y))^2}{|x-y|^{N+2s_0}}\ dxdy\leq\liminf_{n\to\infty}\int_{\Omega}\int_{\Omega}\frac{(u_{s_n}(x)-u_{s_n}(y))^2}{|x-y|^{N+2s_n}}\ dxdy\\
		&=\liminf_{n\to\infty}|u_{s_n}|^2_{H^{s_n}(\Omega)}=\frac{2}{C_{N,s_0}}\|u_{*}\|_{L^2(\Omega)}\|f\|_{L^2(\Omega)}<\infty.
		\end{align}
		This implies that $u_{*}\in H^{s_0}(\Omega)$. We recall that in \eqref{z6}, we have used \eqref{z3} and \eqref{z5}.
		
		On the other hand, for all $\phi\in C^{\infty}_0(\ov\Omega)$, we have
		\begin{align*}
		\langle f,\phi\rangle_{L^2(\Omega)}=\int_{\Omega}f\phi\ dx&=\lim\limits_{n\to\infty}\cE_{s_n}(u_{s_n},\phi)=\lim\limits_{n\to\infty}\int_{\Omega}u_{s_n}(-\Delta)^{s_n}_{\Omega}\phi\ dx\\
		&=\int_{\Omega}u_{*}(-\Delta)^{s_0}_{\Omega}\phi\ dx=\cE_{s_0}(u_{*},\phi).
		\end{align*}
		This shows that $u_{*}\in H^{s_0}(\Omega)$ with $\int_{\Omega}u_{*}\ dx=0$ distributionaly solves the Poisson problem $(-\Delta)^{s_0}_{\Omega}u_{*}=f$ in $\Omega$. Recalling that $u_{s_0}\in H^{s_0}(\Omega)$ with $\int_{\Omega}u_{s_0}\ dx=0$ is the unique weak (distributional) solution to the Poisson problem $(-\Delta)^{s_0}_{\Omega}u_{s_0}=f$ in $\Omega$, we find that $u_{*}\equiv u_{s_0}$, as wanted.\\
		
		\textbf{Step 2.} We show that the solution map $(0,1)\to L^2(\Omega),~s\mapsto u_s$ is right differentiable.\\
		
		Fix $s\in (0,1)$ and define
		\begin{equation}\label{v-sigma}
		v_{\sigma}=\frac{u_{s+\sigma}-u_s}{\sigma}.
		\end{equation}
		Here, $u_{s+\sigma}$ is the unique weak solution of \eqref{poisson-problem-for-regional-fractional-laplacian} with $s$ replaced by $s+\sigma$. We wish first to study the asymptotic behavior of $v_{\sigma}$ as $\sigma\rightarrow0^+$. 
		
		For all $\phi\in C^{\infty}_0(\overline{\Omega})$,
		\begin{align*}
		\cE_s(u_s,\phi)=\int_{\Omega}f\phi\ dx=\cE_{s+\sigma}(u_{s+\sigma},\phi)=\cE_s(u_s-u_{s+\sigma},\phi)+\cE_s(u_{s+\sigma},\phi)
		\end{align*}
		that is
		\begin{align}\label{n0}
		\cE_s(u_s-u_{s+\sigma},\phi)=\cE_{s+\sigma}(u_{s+\sigma},\phi)-\cE_s(u_{s+\sigma},\phi).
		\end{align}
		Now,
		\begin{align}\label{n1}
		\nonumber&\cE_{s+\sigma}(u_{s+\sigma},\phi)-\cE_s(u_{s+\sigma},\phi)\\
		\nonumber&=\frac{C_{N,s+\sigma}}{2}\int_{\Omega}\int_{\Omega}\frac{(u_{s+\sigma}(x)-u_{s+\sigma}(y))(\phi(x)-\phi(y))}{|x-y|^{N+2s+2\sigma}}\ dxdy\\
		\nonumber&\ \ \ \ \ \ \ \ \ \ \ \ \ \ \ -\frac{C_{N,s}}{2}\int_{\Omega}\int_{\Omega}\frac{(u_{s+\sigma}(x)-u_{s+\sigma}(y))(\phi(x)-\phi(y))}{|x-y|^{N+2s}}\ dxdy\\
		\nonumber&=\frac{C_{N,s+\sigma}-C_{N,s}}{2}\int_{\Omega}\int_{\Omega}\frac{(u_{s+\sigma}(x)-u_{s+\sigma}(y))(\phi(x)-\phi(y))}{|x-y|^{N+2s+2\sigma}}\ dxdy\\
		\nonumber&+\frac{C_{N,s}}{2}\int_{\Omega}\int_{\Omega}\Big(\frac{1}{|x-y|^{N+2s+2\sigma}}-\frac{1}{|x-y|^{N+2s}}\Big)(u_{s+\sigma}(x)-u_{s+\sigma}(y))(\phi(x)-\phi(y))\ dxdy\\
		\nonumber&=\frac{1}{C_{N,s+\sigma}}\times(C_{N,s+\sigma}-C_{N,s}) \cE_{s+\sigma}(u_{s+\sigma},\phi)\\
		\nonumber&\ \ \ \ \ \ \ +\frac{C_{N,s}}{2}\int_{\Omega}\int_{\Omega}\Big(|x-y|^{-2\sigma}-1\Big)\frac{(u_{s+\sigma}(x)-u_{s+\sigma}(y))(\phi(x)-\phi(y))}{|x-y|^{N+2s}}\ dxdy\\
		\nonumber&=\frac{1}{C_{N,s+\sigma}}\times(C_{N,s+\sigma}-C_{N,s})\int_{\Omega}f\phi\ dx\\
		&\ \ \ \ \ \ \ +\frac{C_{N,s}}{2}\int_{\Omega}\int_{\Omega}\Big(|x-y|^{-2\sigma}-1\Big)\frac{(u_{s+\sigma}(x)-u_{s+\sigma}(y))(\phi(x)-\phi(y))}{|x-y|^{N+2s}}\ dxdy.
		\end{align}
		Next, we write
		\begin{align}\label{n2}
		\nonumber|x-y|^{-2\sigma}-1&=\exp(-2\sigma\log|x-y|)-1=-2\sigma\log|x-y|\int_{0}^{1}\exp(-2t\sigma\log|x-y|)\ dt\\
		&=-2\sigma\psi_{\sigma}(x,y)\log|x-y|\quad\text{with}~~\psi_{\sigma}(x,y)=\int_{0}^{1}\exp(-2t\sigma\log|x-y|)\ dt.
		\end{align}
		Plugging \eqref{n2} into \eqref{n1}, we get
		\begin{align}\label{n3}
		\nonumber&\cE_{s+\sigma}(u_{s+\sigma},\phi)-\cE_s(u_{s+\sigma},\phi)\\
		\nonumber&=\frac{1}{C_{N,s+\sigma}}\times(C_{N,s+\sigma}-C_{N,s})\int_{\Omega}f\phi\ dx\\
		&-\sigma C_{N,s}\int_{\Omega}\int_{\Omega}\frac{(u_{s+\sigma}(x)-u_{s+\sigma}(y))(\phi(x)-\phi(y))}{|x-y|^{N+2s}}\psi_{\sigma}(x,y)\log|x-y|\ dxdy.
		\end{align}
		Equations \eqref{n0} and \eqref{n3} yield
		\begin{align}\label{n4}
		\nonumber\cE_s(v_{\sigma},\phi)&=-\frac{1}{C_{N,s+\sigma}}\times\frac{C_{N,s+\sigma}-C_{N,s}}{\sigma}\int_{\Omega}f\phi\ dx\\
		&+ C_{N,s}\int_{\Omega}\int_{\Omega}\frac{(u_{s+\sigma}(x)-u_{s+\sigma}(y))(\phi(x)-\phi(y))}{|x-y|^{N+2s}}\psi_{\sigma}(x,y)\log|x-y|\ dxdy
		\end{align}
		for all $\phi\in C^{\infty}_0(\ov\Omega)$. Now, by density, there is $\phi_n\in C^{\infty}_0(\ov\Omega)$ such that $\phi_n\to v_{\sigma}$ in $H^{s+\epsilon}(\Omega)$ for $\epsilon>0$. Moreover, from \eqref{n3},
		\begin{align}\label{n5}
		\nonumber\cE_s(v_{\sigma},\phi_n)&=-\frac{1}{C_{N,s+\sigma}}\times\frac{C_{N,s+\sigma}-C_{N,s}}{\sigma}\int_{\Omega}f\phi_n\ dx\\
		&+ C_{N,s}\int_{\Omega}\int_{\Omega}\frac{(u_{s+\sigma}(x)-u_{s+\sigma}(y))(\phi_n(x)-\phi_n(y))}{|x-y|^{N+2s}}\psi_{\sigma}(x,y)\log|x-y|\ dxdy.
		\end{align}
		Using Cauchy-Schwarz inequality,
		\begin{align}\label{n6}
		\nonumber&|\cE_s(v_{\sigma},\phi_n)-\cE_s(v_{\sigma},v_{\sigma})|\\
		\nonumber&\leq\frac{C_{N,s}}{2}\int_{\Omega}\int_{\Omega}\frac{|v_{\sigma}(x)-v_{\sigma}(y)||(\phi_n(x)-\phi_n(y))-(v_{\sigma}(x)-v_{\sigma}(y))|}{|x-y|^{N+2s}}\ dxdy\\
		&\leq\frac{C_{N,s}}{2}|v_{\sigma}|_{H^s(\Omega)}|\phi_n-v_{\sigma}|_{H^s(\Omega)}\to0\quad\quad\text{as}~~n\to\infty.
		\end{align}
		This implies that
		\begin{equation}\label{n7}
		\cE_s(v_{\sigma},\phi_n)\to\cE_s(v_{\sigma},v_{\sigma})\quad\quad\text{as}~~n\to\infty.
		\end{equation}
		Since also $\phi_n\to v_{\sigma}$ in $L^{2}(\Omega)$, thanks to Poincar\'{e} ineqality, then
		\begin{equation}\label{n8}
		\int_{\Omega}f\phi_n\ dx\to\int_{\Omega}fv_{\sigma}\ dx.
		\end{equation}
		On the other hand, using that 
		\begin{equation}\label{n9}
		|\psi_{\sigma}(x,y)|\leq\max\{1,\exp(-2\sigma\log|x-y|)\}
		\end{equation}
		then applying again Cauchy-Schwarz inequality, one can also show that
		\begin{align}\label{n10}
		\nonumber&\int_{\Omega}\int_{\Omega}\frac{(u_{s+\sigma}(x)-u_{s+\sigma}(y))(\phi_n(x)-\phi_n(y))}{|x-y|^{N+2s}}\psi_{\sigma}(x,y)\log|x-y|\ dxdy\\
		&\to\int_{\Omega}\int_{\Omega}\frac{(u_{s+\sigma}(x)-u_{s+\sigma}(y))(v_{\sigma}(x)-v_{\sigma}(y))}{|x-y|^{N+2s}}\psi_{\sigma}(x,y)\log|x-y|\ dxdy~~~~\text{as}~~n\to\infty.
		\end{align}
		Combining \eqref{n7}, \eqref{n8} and \eqref{n10}, then from \eqref{n5} it follows that
		\begin{align}\label{n11}
		\nonumber\cE_s(v_{\sigma},v_{\sigma})&=-\frac{1}{C_{N,s+\sigma}}\times\frac{C_{N,s+\sigma}-C_{N,s}}{\sigma}\int_{\Omega}fv_{\sigma}\ dx\\
		&+ C_{N,s}\int_{\Omega}\int_{\Omega}\frac{(u_{s+\sigma}(x)-u_{s+\sigma}(y))(v_{\sigma}(x)-v_{\sigma}(y))}{|x-y|^{N+2s}}\psi_{\sigma}(x,y)\log|x-y|\ dxdy.
		\end{align}
		From \eqref{n11}, we write
		\begin{align}\label{n12}
		\nonumber\frac{C_{N,s}}{2}|v_{\sigma}|^2_{H^s(\Omega)}&=\cE_s(v_{\sigma},v_{\sigma})\leq\frac{1}{C_{N,s+\sigma}}\Big|\frac{C_{N,s+\sigma}-C_{N,s}}{\sigma}\Big|\int_{\Omega}|f||v_{\sigma}|\ dx\\
		&+ C_{N,s}\int_{\Omega}\int_{\Omega}\frac{|u_{s+\sigma}(x)-u_{s+\sigma}(y)||v_{\sigma}(x)-v_{\sigma}(y)|}{|x-y|^{N+2s}}|\psi_{\sigma}(x,y)||\log|x-y||\ dxdy.
		\end{align}
		Since the map $(0,1)\ni s\mapsto C_{N,s}$ is of class $C^1$, then
		\begin{align}\label{n13}
		\frac{1}{C_{N,s+\sigma}}\Big|\frac{C_{N,s+\sigma}-C_{N,s}}{\sigma}\Big|\leq\frac{|\partial_sC_{N,s}|}{C_{N,s}}+o(1)\quad\quad\text{as}~~\sigma\to0^+.
		\end{align}
		Now, H\"{o}lder inequality and Poincar\'{e} inequality (see \eqref{Poincare-inequality}) yield
		\begin{align}\label{n14}
		\int_{\Omega}|f||v_{\sigma}|\ dx\leq\|f\|_{L^2(\Omega)}\|v_{\sigma}\|_{L^2(\Omega)}\leq C(N,s,\Omega,\|f\|_{L^2(\Omega)})|v_{\sigma}|_{H^s(\Omega)}.
		\end{align}
		On the other hand, from \eqref{n9}, we have
		\begin{align*}
		&\int_{\Omega}\int_{\Omega}\frac{|u_{s+\sigma}(x)-u_{s+\sigma}(y)||v_{\sigma}(x)-v_{\sigma}(y)|}{|x-y|^{N+2s}}|\psi_{\sigma}(x,y)||\log|x-y||\ dxdy\\
		&\leq\int_{\Omega}\int_{\Omega}\frac{|u_{s+\sigma}(x)-u_{s+\sigma}(y)||v_{\sigma}(x)-v_{\sigma}(y)|}{|x-y|^{N+2s}}|\log|x-y||\ dxdy\\
		&\ \ \ \ +\int_{\Omega}\int_{\Omega}\frac{|u_{s+\sigma}(x)-u_{s+\sigma}(y)||v_{\sigma}(x)-v_{\sigma}(y)|}{|x-y|^{N+2s+2\sigma}}|\log|x-y||\ dxdy.
		\end{align*}
		To estimate the first term on the RHS of the above inequality, we use Cauchy-Schwarz inequality together with the Logarithmic decay \eqref{logarithmic-decays}:
		\begin{align}\label{n15}
		\nonumber&\int_{\Omega}\int_{\Omega}\frac{|u_{s+\sigma}(x)-u_{s+\sigma}(y)||v_{\sigma}(x)-v_{\sigma}(y)|}{|x-y|^{N+2s}}|\log|x-y||\ dxdy\\
		\nonumber&\leq\Big(\int_{\Omega}\int_{\Omega}\frac{|u_{s+\sigma}(x)-u_{s+\sigma}(y)|^2}{|x-y|^{N+2s}}|\log|x-y||^2\ dxdy\Big)^{1/2}|v_{\sigma}|_{H^s(\Omega)}\\
		\nonumber&\leq\Bigg(\frac{1}{(e\epsilon_0)^2}\int_{\Omega}\int_{\Omega}\frac{|u_{s+\sigma}(x)-u_{s+\sigma}(y)|^2}{|x-y|^{N+2s}}(|x-y|^{-2\epsilon_0}+|x-y|^{2\epsilon_0})\ dxdy\Bigg)^{1/2}|v_{\sigma}|_{H^s(\Omega)}\\
		\nonumber&\leq c(\epsilon_0)\Bigg(\int_{\Omega}\int_{\Omega}\frac{|u_{s+\sigma}(x)-u_{s+\sigma}(y)|^2}{|x-y|^{N+2s+2\epsilon_0}}\ dxdy+d^{2\epsilon_0}_{\Omega}|u_{s+\sigma}|^2_{H^s(\Omega)}\Bigg)^{1/2}|v_{\sigma}|_{H^s(\Omega)}\\
		&=c(\epsilon_0)\Big(|u_{s+\sigma}|^2_{H^{s+\epsilon_0}(\Omega)}+d^{2\epsilon_0}_{\Omega}|u_{s+\sigma}|^2_{H^s(\Omega)}\Big)^{1/2}|v_{\sigma}|_{H^s(\Omega)}.
		\end{align}
		By Proposition \ref{higher-sobolev-regularity} there exist $K_1, K_2>0$ independent on $\sigma$ such that
		\begin{equation}\label{n16}
		|u_{s+\sigma}|_{H^{s+\epsilon_0}(\Omega)}\leq K_1\quad\quad\text{and}\quad\quad |u_{s+\sigma}|_{H^s(\Omega)}\leq K_2\quad\text{for}~~\sigma~~\text{sufficiently small}.
		\end{equation}
		Combining this with \eqref{n15}, we obtain that
		\begin{equation}\label{n20}
		\int_{\Omega}\int_{\Omega}\frac{|u_{s+\sigma}(x)-u_{s+\sigma}(y)||v_{\sigma}(x)-v_{\sigma}(y)|}{|x-y|^{N+2s}}|\log|x-y||\ dxdy\leq c|v_{\sigma}|_{H^s(\Omega)}.
		\end{equation}
		By a similar argument as above, we also obtain the following bound
		\begin{equation}\label{n21}
		\int_{\Omega}\int_{\Omega}\frac{|u_{s+\sigma}(x)-u_{s+\sigma}(y)||v_{\sigma}(x)-v_{\sigma}(y)|}{|x-y|^{N+2s+2\sigma}}|\log|x-y||\ dxdy\leq c|v_{\sigma}|_{H^s(\Omega)}.
		\end{equation}
		Combining \eqref{n13}, \eqref{n14}, \eqref{n20}, and \eqref{n21}, we find from \eqref{n12} that
		\begin{equation}\label{n22}
		|v_{\sigma}|_{H^s(\Omega)}\leq c.
		\end{equation}
		In other words, $v_{\sigma}$ is uniformly bounded in $H^s(\Omega)$ with respect to $\sigma$. Therefore, after passing to a subsequence, there is $w_s\in H^s(\Omega)$ such that 
		\begin{equation}\label{n23}
		\begin{aligned}
		&v_{\sigma}\rightharpoonup w_s\quad\text{weakly in}~~H^s(\Omega),\\
		&v_{\sigma}\rightarrow w_s\quad\text{strongly in}~~L^2(\Omega),\\
		&v_{\sigma}\rightarrow w_s\quad\text{a.e. in}~~\Omega. 
		\end{aligned}
		\end{equation}
		In particular, $\int_{\Omega}w_s\ dx=0$ since does $v_{\sigma}$. \\
		
		To obtain the right-differentiability of the solution map $s\mapsto u_s$, it suffices to show that $w_s$ is unique as a limit of the whole sequence $v_{\sigma}$.
		
		First of all, from \eqref{n4}, thanks to Proposition \ref{higher-sobolev-regularity} and Dominated Convergence Theorem, we deduce that $w_s$ solves
		\begin{align}\label{n29}
		\nonumber\cE_s(w_s,\phi)&=-\frac{\partial_sC_{N,s}}{C_{N,s}}\int_{\Omega}f\phi\ dx\\
		&+ C_{N,s}\int_{\Omega}\int_{\Omega}\frac{(u_s(x)-u_s(y))(\phi(x)-\phi(y))}{|x-y|^{N+2s}}\log|x-y|\ dxdy
		\end{align}
		for all $\phi\in C^{\infty}_0(\ov\Omega)$.
		
		Let denote by $v_{\sigma_k}$ the corresponding subsequence of $v_{\sigma}$ for which \eqref{n23} holds. Consider now another subsequence $v_{\sigma_i}$ with $v_{\sigma_i}\to \overline{w}_s$ for some $\overline{w}_s\in H^s(\Omega)$ with $\int_{\Omega}\overline{w}_s\ dx=0$. We wish to prove that prove that $w_s=\overline{w}_s$. Let set $W_s=w_s-\overline{w}_s$. Then, in particular, $\int_{\Omega}W_s\ dx=0$. Moreover, for all $\phi\in C^{\infty}_0(\ov\Omega)$, 
		\begin{align}\label{n30}
		\nonumber\cE_s(W_s,\phi)&=\cE_s(w_s-\overline{w}_s,\phi)=\cE_s(w_s,\phi)-\cE_s(\overline{w}_s,\phi)=\cE_s(w_s,\phi)-\lim\limits_{i\to\infty}\cE_s(v_{\sigma_i},\phi)\\
		\nonumber&=\cE_s(w_s,\phi)-\lim\limits_{i\to\infty}\cE_s(v_{\sigma_i}-w_s+w_s,\phi)=-\lim\limits_{i\to\infty}\cE_s(v_{\sigma_i}-w_s,\phi)\\
		&=-\lim\limits_{i\to\infty}\lim\limits_{k\to\infty}\cE_s(v_{\sigma_i}-v_{\sigma_k},\phi).
		\end{align}
		From \eqref{n4}, one gets 
		\begin{align*}
		\cE_s(v_{\sigma_i}-v_{\sigma_k},\phi)&=\Bigg(-\frac{1}{C_{N,s+\sigma_i}}\times\frac{C_{N,s+\sigma_i}-C_{N,s}}{\sigma_i}+\frac{1}{C_{N,s+\sigma_k}}\times\frac{C_{N,s+\sigma_k}-C_{N,s}}{\sigma_k}\Bigg)\int_{\Omega}f\phi\ dx\\
		&+ C_{N,s}\int_{\Omega}\int_{\Omega}\frac{(u_{s+\sigma_i}(x)-u_{s+\sigma_i}(y))(\phi(x)-\phi(y))}{|x-y|^{N+2s}}\psi_{\sigma_i}(x,y)\log|x-y|\ dxdy\\
		&-C_{N,s}\int_{\Omega}\int_{\Omega}\frac{(u_{s+\sigma_k}(x)-u_{s+\sigma_k}(y))(\phi(x)-\phi(y))}{|x-y|^{N+2s}}\psi_{\sigma_k}(x,y)\log|x-y|\ dxdy.
		\end{align*}
		Using that $s\mapsto C_{N,s}$ is of class $C^1$, the fact that $s\mapsto u_s$ is continuous and Proposition \ref{higher-sobolev-regularity}, the Dominated Convergence Theorem yields
		\begin{equation}
		\lim\limits_{i\to\infty}\lim\limits_{k\to\infty}\cE_s(v_{\sigma_i}-v_{\sigma_k},\phi)=0.
		\end{equation}
		Consequently, one gets from \eqref{n30} that
		\begin{equation}\label{n31}
		\cE_s(W_s,\phi)=0.
		\end{equation}
		By density, \eqref{n31} also holds with $\phi$ replaced by $W_s$, that is,
		\begin{equation}\label{n32}
		\cE_s(W_s,W_s)=0.
		\end{equation}
		This implies that $W_s=const$. Morever, since also $\int_{\Omega}W_s\ dx=0$, then we get that $W_s=0$, that is, $w_s=\overline{w}_s$ as wanted. In conclusion, the solution map $s\mapsto u_s$ is right-differentiable with $\partial_s^+u_s=w_s$,  solving uniquely \eqref{n29}.\\
		
		\textbf{Step 3.} We establishe the continuity of the map $s\mapsto\partial_s^+u_s=w_s$. The proof of this is similar to that in \textbf{Step 1} and we include it for the sake of completeness. 
		Fix again $s_0\in(0,1)$. Let $\delta\in(0,s_0)$ and $(s_n)_n\subset(s_0-\delta,1)$ be a sequence such that $s_n\to s_0$. By putting also $s':=\inf_{n\in\N}s_n>s_0-\delta$, then as in \eqref{z2} one get that
		\begin{equation}\label{w1}
		|w_{s_n}|^2_{H^{s'}(\Omega)}\leq c|w_{s_n}|^2_{H^{s_n}(\Omega)}.
		\end{equation}
		Now, from \eqref{n29}, it follows that
		\begin{align}\label{n35}
		\nonumber\cE_{s_n}(w_{s_n},w_{s_n})&=-\frac{\partial_{s_n}C_{N,s_n}}{C_{N,s_n}}\int_{\Omega}fw_{s_n}\ dx\\
		&+ C_{N,s_n}\int_{\Omega}\int_{\Omega}\frac{(u_{s_n}(x)-u_{s_n}(y))(w_{s_n}(x)-w_{s_n}(y))}{|x-y|^{N+2s_n}}\log|x-y|\ dxdy,
		\end{align}
		that is,
		\begin{align}\label{n36}
		\nonumber|w_{s_n}|^2_{H^{s_n}(\Omega)}&\leq\Big|\frac{2\partial_{s_n}C_{N,s_n}}{C^2_{N,s_n}}\Big|\int_{\Omega}|f||w_{s_n}|\ dx\\
		&+ 2\int_{\Omega}\int_{\Omega}\frac{|u_{s_n}(x)-u_{s_n}(y)||w_{s_n}(x)-w_{s_n}(y)|}{|x-y|^{N+2s_n}}|\log|x-y||\ dxdy.
		\end{align}
		Now, using that $s\mapsto C_{N,s}$ is of class $C^1$, then
		\begin{align}\label{n37}
		\Big|\frac{2\partial_{s_n}C_{N,s_n}}{C^2_{N,s_n}}\Big|\leq\Big|\frac{2\partial_{s_0}C_{N,s_0}}{C^2_{N,s_0}}\Big|+o(1)\quad\text{as}~~n\to\infty.
		\end{align}
		By H\"{o}lder inequality and Poincar\'{e} inequality \eqref{Poincare-inequality}, we get
		\begin{align}\label{n38}
		\int_{\Omega}|f||w_{s_n}|\ dx\leq c|w_{s_n}|_{H^{s_n}(\Omega)} 
		\end{align}
		On the other hand, Cauchy-Schwarz inequality together with \eqref{logarithmic-decays} yield
		\begin{align}\label{n39}
		\nonumber&\int_{\Omega}\int_{\Omega}\frac{|u_{s_n}(x)-u_{s_n}(y)||w_{s_n}(x)-w_{s_n}(y)|}{|x-y|^{N+2s_n}}|\log|x-y||\ dxdy\\
		&\leq\Big(c_{\epsilon_0}|u_{s_n}|^2_{H^{s_n+\epsilon_0}(\Omega)}+\tilde{c}_{\epsilon_0}|u_{s_n}|^2_{H^{s_n}(\Omega)}\Big)|w_{s_n}|_{H^{s_n}(\Omega)}.
		\end{align}
		By Proposition \ref{higher-sobolev-regularity}, we find that
		\begin{equation}
		|u_{s_n}|_{H^{s_n+\epsilon_0}(\Omega)}\leq C_1\quad\quad\text{and}\quad\quad|u_{s_n}|_{H^{s_n}(\Omega)}\leq C_2\quad\text{as}~~n\to\infty.
		\end{equation}
		Taking this into account, we get from \eqref{n39} that
		\begin{equation}\label{w2}
		\int_{\Omega}\int_{\Omega}\frac{|u_{s_n}(x)-u_{s_n}(y)||w_{s_n}(x)-w_{s_n}(y)|}{|x-y|^{N+2s_n}}|\log|x-y||\ dxdy\leq C|w_{s_n}|_{H^{s_n}(\Omega)}
		\end{equation}
		for $n$ sufficiently large.
		
		It follows from \eqref{n37}, \eqref{n38}, \eqref{w2} and \eqref{w1} that
		\begin{align}\label{n40}
		|w_{s_n}|_{H^{s'}(\Omega)}\leq|w_{s_n}|_{H^{s_n}(\Omega)}\leq c\quad\text{for}~~n~~\text{sufficiently large}.
		\end{align}
		This means that $w_{s_n}$ is uniformly bounded in $H^{s'}(\Omega)$ with respect to $n$. Therefore, up to a subsequence, there is $w_{*}\in H^{s'}(\Omega)$ such that
		\begin{equation}\label{n41}
		\begin{aligned}
		&w_{s_n}\rightharpoonup w_{*}\quad\text{weakly in}~~H^{s'}(\Omega),\\
		&w_{s_n}\rightarrow w_{*}\quad\text{strongly in}~~L^2(\Omega),\\
		&w_{s_n}\rightarrow w_{*}\quad\text{a.e. in}~~\Omega. 
		\end{aligned}
		\end{equation}
		In particular, we have $\int_{\Omega}w_{*}\ dx=0$. Next, we show that $w_{*}\equiv w_{s_0}$.
		
		By Fatou's Lemma,
		\begin{align*}
		|w_{*}|_{H^{s_0}(\Omega)}&=\Big(\int_{\Omega}\int_{\Omega}\frac{(w_{*}(x)-w_{*}(y))^2}{|x-y|^{N+2s_0}}\ dxdy\Big)^{1/2}\\
		&\leq\liminf_{n\to\infty}\Big(\int_{\Omega}\int_{\Omega}\frac{(w_{s_n}(x)-w_{s_n}(y))^2}{|x-y|^{N+2s_n}}\ dxdy\Big)^{1/2}\\
		&=\liminf_{n\to\infty}|w_{s_n}|_{H^{s_n}(\Omega)}\leq C<\infty.
		\end{align*}
		This implies that $w_{*}\in H^{s_0}(\Omega)$. Notice that we have used \eqref{n36}, \eqref{n37}, \eqref{n38}, \eqref{n39} and Proposition \ref{higher-sobolev-regularity}.
		
		On the other hand, for all $\phi\in C^{\infty}_0(\ov\Omega)$, we have
		\begin{align}\label{w3}
		\nonumber&\cE_{s_0}(w_{*},\phi)=\int_{\Omega}w_{*}(-\Delta)^{s_0}_{\Omega}\phi\ dx=\lim\limits_{n\to\infty}\int_{\Omega}w_{s_n}(-\Delta)^{s_n}_{\Omega}\phi\ dx=\lim\limits_{n\to\infty}\cE_{s_n}(w_{s_n},\phi)\\
		&=-\lim\limits_{n\to\infty}\frac{\partial_{s_n}C_{N,s_n}}{C_{N,s_n}}\int_{\Omega}f\phi\ dx+\lim\limits_{n\to\infty} C_{N,s_n}\int_{\Omega}\int_{\Omega}\frac{(u_{s_n}(x)-u_{s_n}(y))(\phi(x)-\phi(y))}{|x-y|^{N+2s_n}}\log|x-y|\ dxdy.
		\end{align}
		In \eqref{w3}, we have used \eqref{n29}. Now, by \textbf{Step 1} and Proposition \ref{higher-sobolev-regularity}, one obtains from \eqref{w3} that
		\begin{align}\label{n42}
		\cE_{s_0}(w_{*},\phi)=-\frac{\partial_{s_0}C_{N,s_0}}{C_{N,s_0}}\int_{\Omega}f\phi\ dx
		+ C_{N,s_0}\int_{\Omega}\int_{\Omega}\frac{(u_{s_0}(x)-u_{s_0}(y))(\phi(x)-\phi(y))}{|x-y|^{N+2s_0}}\log|x-y|\ dxdy
		\end{align}
		for all $\phi\in C^{\infty}_0(\ov\Omega)$. \\
		
		Since by \textbf{Step 2} $w_{s_0}\in H^{s_0}(\Omega)$ with $\int_{\Omega}w_{s_0}\ dx=0$ is the unique solution to \eqref{n42}, then one finds that $w_{*}\equiv w_{s_0}$. This yields the continuity of the map $s\mapsto w_s$ and this concudes the proof of \textbf{Step 3}.\\
		
		In summary, from \textbf{Steps 1, 2} and \textbf{3}, we have shown that
		\begin{itemize}
			\item [$(i)$] $s\mapsto u_s$ is continuous;
			\item [$(ii)$] $\partial_s^+u_s$ exists in $L^2(\Omega)$ for all $s\in (0,1)$;
			\item [$(iii)$] The map $(0,1)\to L^2(\Omega),~s\mapsto\partial_s^+u_s$ is continuous.
		\end{itemize}
		Therefore, by Proposition \ref{diff-of-the-curve}, we conclude that the solution map $(0,1)\to L^2(\Omega),~s\mapsto\partial_s^+u_s$ is continuously differentiable with $\partial_su_s=\partial_s^+u_s$. Moreover, from \eqref{n29}, we have that $w_s=\partial_su_s$ solves in weak sense the equation
		\begin{equation}
		(-\Delta)^s_{\Omega}w_s=M^s_{\Omega}u_s\quad\quad\text{in}~~\Omega
		\end{equation}
		with
		\begin{equation*}
		M^s_{\Omega}u(x)=-\frac{\partial_sC_{N,s}}{C_{N,s}}f(x)+2C_{N,s}P.V.\int_{\Omega}\frac{(u(x)-u(y))}{|x-y|^{N+2s}}\log|x-y|\ dy,~~x\in\Omega.
		\end{equation*}
		
	\end{proof}

	\section{Eigenvalues problem case}\label{eigenvalue-problem-case}
	The aim of this section is to study \eqref{poisson-problem-for-regional-fractional-laplacian} when $f=\lambda_su_s$ i.e., the eigenvalues problem
	\begin{equation}\label{eigenvalues-problem}
	(-\Delta)^s_{\Omega}u_s=\lambda_su_s\quad\text{in}\quad\Omega.
	\end{equation}
	More precisely, we discuss the $s$-dependence of the map $s\mapsto \lambda_{1,s}$ where $\lambda_{1,s}$ is the first nontrivial eigenvalue of $(-\Delta)^s_{\Omega}$. We notice that equation \eqref{eigenvalues-problem} is understood in weak sense. Here and throughout the end of this section, we fix $\Omega$ as a bounded domain with $C^{1,1}$ boundary.
	
	Let
	\begin{equation}
	0<\lambda_{1,s}\leq\lambda_{2,s}\leq\cdots\leq\lambda_{k,s}\leq\cdots,
	\end{equation}
	be the sequence of eigenvalues (counted with multiplicity) of $(-\Delta)^s_{\Omega}$ in $\Omega$ with corresponding eigenfunctions 
	$$\phi_{1,s}, \phi_{2,s},\dots, \phi_{k,s},\dots.$$ 
	It is known that the system $\{\phi_{i,s}\}_i$ form an $L^2$-orthonormal basis. Variationnaly, we have \begin{equation}\label{k-th-eigenvalue-of-regional-fractional-laplacian}
	\lambda_{k,s}:=\inf_{V\in V^s_k}\sup_{\phi\in S_{V}}\cE_s(\phi,\phi),
	\end{equation}
	where $V^s_k:=\{V\subset\mathbb{X}^s(\Omega):\dim V=k\}$ and $S_{V}:=\{\phi\in V:\|\phi\|_{L^2(\Omega)}=1\}$ for all $V\in V^s_k$. However, when $k=1$ then $\lambda_{1,s}$ is simply  characterized by (see e.g., \cite[Theorem 3.1]{del2015first})
	\begin{equation}\label{first-eigenvalue-of-regional-fractional-laplacian}
	\lambda_{1,s}:=\inf\{\cE_s(\phi,\phi):\phi\in\mathbb{X}^s(\Omega),\|\phi\|_{L^2(\Omega)}=1\}.
	\end{equation}
	
	In this section, we wish to study the differentiability of the map $(0,1)\ni s\mapsto \lambda_{1,s}$. As first remark, we know that the first nontrivial eigenvalue of $(-\Delta)^s_{\Omega}$ is not in general \textit{simple}. Therefore, the main focus here is one-sided differentiability.
	
	In what follows, we discuss the right differentiability of the map $s\mapsto\lambda_{1,s}$ stated in Theorem \ref{differentiability-of-eigenvalues-0}. For the reader's convenience, we restate it in the following. Here and throughout the end of this Section, we use respectively $\lambda_s$ and $u_s$ for $\lambda_{1,s}$ and $u_{1,s}$ to alleviate the notation. 
	\begin{thm}\label{differentiability-of-eigenvalues}
		Regarded as function of $s$, $\lambda_s$ is right differentiable on $(0,1)$ and
		\begin{equation}\label{C-1-regularity-of-eigenvalues}
		\partial^+_s\lambda_s:=\lim\limits_{\sigma\rightarrow0^+}\frac{\lambda_{s+\sigma}-\lambda_s}{\sigma}=\inf\{J_s(u):u\in \M_s\}
		\end{equation}
		where
		\begin{equation}
		J_s(u)=\frac{\partial_sC_{N,s}}{C_{N,s}}\lambda_s-C_{N,s}\int_{\Omega}\int_{\Omega}\frac{(u(x)-u(y))^2}{|x-y|^{N+2s}}\log|x-y|\ dxdy
		\end{equation}
		and $\M_s$ the set of $L^2$-normalized eigenfunctions of $(-\Delta)^s_{\Omega}$ corresponding to $\lambda_s$. 
		Moreover, the infimum in \eqref{C-1-regularity-of-eigenvalues} is attained.
	\end{thm}
	We now collect some partial results needed for the proof of Theorem \ref{differentiability-of-eigenvalues}. In the sequel, we prove the following two lemmas in the same spirit as Theorem 1.3 and Lemma 2.1 in \cite{deblassie2005alpha}.
	\begin{lemma}\label{continuity-of-quadratic-form-with-respect-to-s}
		Let $\phi,\psi\in C^{\infty}(\overline{\Omega})$. Then, regarded as function of $s$,
		\begin{equation*}
		\cE_s(\phi,\psi):(0,1)\rightarrow\R
		\end{equation*}
		is continuous on $(0,1)$.
	\end{lemma}
	\begin{proof}
		It suffices to show that
		\begin{equation*}
		\lim\limits_{\alpha\rightarrow s}\cE_{\alpha}(\phi,\psi)=\cE_s(\phi,\psi).
		\end{equation*}
		Let $\alpha\in(s-\delta,s+\delta)$ where $\delta=\frac{1}{4}\min\{1-s,s\}$. Then,
		\begin{align*}
		\frac{(\phi(x)-\phi(y))(\psi(x)-\psi(y))}{|x-y|^{N+2\alpha}}&\leq\frac{|\phi(x)-\phi(y)||\psi(x)-\psi(y)|}{|x-y|^{N+2\alpha}}\\&\leq\|\nabla\phi\|_{L^{\infty}(\Omega)}\|\nabla\psi\|_{L^{\infty}(\Omega)}\frac{1}{|x-y|^{N+2\alpha-2}}\\&\leq C(\phi,\psi)\max\Big\{\frac{1}{|x-y|^{N+2(s-\delta)-2}},\frac{1}{|x-y|^{N+2(s+\delta)-2}}\Big\}\\&=:g_{s,\delta}(x,y).
		\end{align*}
		Using polar coordinates, it is not difficult to see that $g_{s,\delta}$ is integrable on $\Omega\times\Omega$ with
		\begin{equation*}
		\int_{\Omega}\int_{\Omega}|g_{s,\delta}(x,y)|\ dxdy\leq C(\phi,\psi)\Big(\frac{|\Omega||S^{N-1}|}{2(1-s+\delta)}d^{2(1-s+\delta)}_{\Omega}+\frac{|\Omega||S^{N-1}|}{2(1-s-\delta)}d^{2(1-s-\delta)}_{\Omega}\Big),
		\end{equation*}
		and therefore, applying Lebesgue's Dominated Convergence Theorem, we get that
		\begin{equation*}
		\lim\limits_{\alpha\rightarrow s}\cE_{\alpha}(\phi,\psi)=\cE_s(\phi,\psi),
		\end{equation*}
		as needed.
	\end{proof}
	\begin{lemma}\label{right-continuity-of-eigenvalues}
		Let $k\geq1$ and $\lambda_{k,s}$ the $k$-th eigenvalue of $(-\Delta)^s_{\Omega}$ in $\Omega$. Then, regarded as function of $s$, $\lambda_{k,s}$ is continuous on $(0,1)$ for all $k\in\N$.
	\end{lemma}
	\begin{proof}
		The proof is divided in two steps. First one shows the limsup inequality. The second step is to obtain the reverse inequality i.e., the liminf inequality.\\
		
		\textbf{\textit{Step 1}.} We show that
		\begin{equation}\label{limit-sup-for-continuity}
		\limsup_{\alpha\rightarrow s}\lambda_{k,\alpha}\leq\lambda_{k,s}.
		\end{equation}
		Let $\epsilon>0$  and $k\geq1$. Using that $C^{\infty}_0(\overline{\Omega})$ is dense in $\mathbb{X}^s(\Omega)$, there exist $\phi_1,\dots,\phi_k\in C^{\infty}_0(\overline{\Omega})$ such that 
		\begin{equation}\label{density-consequence}
		|\langle\phi_{i,s},\phi_{j,s}\rangle_{L^2(\Omega)}-\langle\phi_{i},\phi_{j}\rangle_{L^2(\Omega)}|\leq\frac{\epsilon}{8k^2}\quad\text{and}\quad|\cE_s(\phi_{i,s},\phi_{j,s})-\cE_s(\phi_{i},\phi_{j})|\leq\frac{\epsilon}{8k^2},
		\end{equation}
		for all $1\leq i,j\leq k$. Now, from Lemma \ref{continuity-of-quadratic-form-with-respect-to-s}, there is $\beta_0>0$ such that for all $\alpha\in (s-\beta_0,s+\beta_0)$,
		\begin{equation}\label{c}
		|\cE_{\alpha}(\phi_i,\phi_j)-\cE_s(\phi_i,\phi_j)|\leq\frac{\epsilon}{8k^2}.
		\end{equation}
		According to \eqref{density-consequence}, we also have
		\begin{equation*}
		|\langle\phi_{i},\phi_j\rangle_{L^2(\Omega)}|<\frac{\epsilon}{8k^2}~(i\neq j)\quad\text{and}\quad 1-\frac{\epsilon}{8k^2}<\|\phi_i\|^2_{L^2(\Omega)}<1+\frac{\epsilon}{8k^2},
		\end{equation*}
		and therefore as in \cite[Section $2$]{deblassie2005alpha}, the familly  $\{\phi_i\}_{i=1,\dots,k}$ is linearily independent. 
		As a consequence, we have by setting in particular $V=\text{span}\{\phi_1,\dots,\phi_k\}$ that
		\begin{equation}\label{d}
		\lambda_{k,\alpha}\leq\sup_{\phi\in S_V}\cE_{\alpha}(\phi,\phi)\leq\cE_{\alpha}(\phi,\phi)+\frac{\epsilon}{4}.
		\end{equation}
		Now, for $\phi\in S_V$, there is a sequence of real numbers $\{a_i\}_{i=1,\dots,k}\subset\R$ satisfying $\sum_{i=1}^{k}a^2_i=1$ such that $\phi=\sum_{i=1}^{k}a_i\phi_i$. Using this and \eqref{c}, we get
		\begin{align*}
		|\cE_{\alpha}(\phi,\phi)-\cE_s(\phi,\phi)|\leq\sum_{i=1}^{k}\sum_{j=1}^{k}|a_i||a_j||\cE_{\alpha}(\phi_i,\phi_j)-\cE_s(\phi_i,\phi_j)|\leq\frac{\epsilon}{4},
		\end{align*}
		i.e.,
		\begin{equation*}
		\cE_{\alpha}(\phi,\phi)\leq\cE_s(\phi,\phi)+\frac{\epsilon}{4}.
		\end{equation*}
		Consequently, we have with \eqref{d} that
		\begin{equation}\label{e}
		\lambda_{k,\alpha}\leq\cE_s(\phi,\phi)+\frac{\epsilon}{2}.
		\end{equation}
		Now, by letting $\psi=\sum_{i=1}^{k}a_i\phi_{i,s}$ and by using \eqref{density-consequence}, we can follows the argument above to show that
		\begin{equation}
		|\cE_s(\psi,\psi)-\cE_s(\phi,\phi)|<\frac{\epsilon}{4}
		\end{equation}
		i.e.,
		\begin{equation}
		\cE_s(\phi,\phi)<\cE_s(\psi,\psi)+\frac{\epsilon}{4}.
		\end{equation}
		Combining this with \eqref{e} and by using also the monotonicity of $\{\lambda_{i,s}\}_i$, we see that
		\begin{align*}
		\lambda_{k,\alpha}&\leq\cE_{\alpha}(\phi,\phi)+\frac{\epsilon}{2}\leq\cE_s(\psi,\psi)+\frac{3\epsilon}{4}\\&\leq\sum_{i=1}^{k}a^2_i\lambda_{i,s}+\frac{3\epsilon}{4}\leq\lambda_{k,s}\sum_{i=1}^{k}a^2_i+\frac{3\epsilon}{4}=\lambda_{k,s}+\frac{3\epsilon}{4}.
		\end{align*}
		Since $\epsilon$ was chosen arbitrarily, we therefore have
		\begin{equation*}
		\limsup_{\alpha\rightarrow s}\lambda_{k,\alpha}\leq\lambda_{k,s},
		\end{equation*}
		as claimed.\\
		
		\textbf{\textit{Step 2}.} We show that
		\begin{equation}\label{limit-inf-for-continuity}
		\liminf_{\alpha\rightarrow s}\lambda_{k,\alpha}\geq\lambda_{k,s}.
		\end{equation}
		To this end, we set $\lambda^{*}_{k,s}:=\liminf_{\alpha\rightarrow s}\lambda_{k,\alpha}$ and let $\alpha_n\in(0,1)$ be a sequence  such that $\alpha_n\to s$ and  $\lambda_{k,\alpha_n}\rightarrow\lambda^{*}_{k,s}$ as $n\rightarrow\infty$. We now choose a system of $L^2$-orthonormal eigenfunctions $\phi_{1,\alpha_n}, \dots,  \phi_{k,\alpha_n}$ associated to $\lambda_{1,\alpha_n}, \dots, \lambda_{k,\alpha_n}$.
		
		By Proposition \ref{higher-sobolev-regularity}, we have that for $n$ sufficiently large,
		\begin{align}\label{f}
		\phi_{j,\alpha_n}~~\text{is uniformly bounded in}~~H^s(\Omega)~~\text{for}~~j=1,\dots,k.
		\end{align}
		
		Therefore, after passing to a subsequence, there exists $e_{j,s}\in H^s(\Omega)$ such that
		\begin{equation*}
		\begin{aligned}
		&\phi_{j,\alpha_n}\rightharpoonup e_{j,s}\quad\text{weakly in}~~H^s(\Omega),\\
		&\phi_{j,\alpha_n}\rightarrow e_{j,s}\quad\text{strongly in}~~L^2(\Omega),~~\text{for}~~j=1, \dots, k,\\
		&\phi_{j,\alpha_n}\rightarrow e_{j,s}\quad\text{a.e. in}~~\Omega,
		\end{aligned}
		\end{equation*}
		which therefore imply that $\int_{\Omega}e_{j,s}\ dx=0.$ Thus, $e_{j,s}\in\mathbb{X}^s(\Omega).$~ Furthermore, by strong convergence in $L^2(\Omega)$, it follows also that $e_{1,s},\dots, e_{k,s}$ form an $L^2$-orthonormal system.
		
		Moreover, for every $j=1,\dots,k,$ we have
		\begin{align*}
		\lambda^{*}_{j,s}\langle e_{j,s},\phi\rangle_{L^2(\Omega)}&=\lim\limits_{n\rightarrow\infty}\lambda_{j,\alpha_n}\langle\phi_{j,\alpha_n},\phi\rangle_{L^2(\Omega)}=\lim\limits_{n\to\infty}\langle \phi_{j,\alpha_n},(-\Delta)^{\alpha_n}_{\Omega}\phi\rangle_{L^2(\Omega)}\\
		&=\langle e_{j,s},(-\Delta)^{s}_{\Omega}\phi\rangle_{L^2(\Omega)}=\cE_s(e_{j,s},\phi)
		\end{align*}
		i.e.,
		$$	\cE_s(e_{j,s},\phi)=\lambda^{*}_{j,s}\langle e_j,\phi\rangle_{L^2(\Omega)}~~\text{for all}~~\phi\in C^{\infty}_0(\overline{\Omega}),$$
		and by density,
		$$	\cE_s(e_{j,s},\phi)=\lambda^{*}_{j,s}\langle e_{j,s},\phi\rangle_{L^2(\Omega)}~~\text{for all}~~\phi\in\mathbb{X}^s(\Omega).$$
		Therefore, $(\lambda^{*}_{j,s})_{j\in\{1,\dots,k\}}$ is an increasing sequence of  eigenvalues of $(-\Delta)^s_{\Omega}$ with corresponding eigenfunctions $(e_{j,s})_{j\in\{1,\dots,k\}}.$ Now, by choosing in particular  $V=\text{span}\{e_{1,s},e_{2,s},\dots,e_{k,s}\}$, we have from \eqref{k-th-eigenvalue-of-regional-fractional-laplacian} that
		\begin{equation}\label{uper-bound-of-lambda-k}
		\lambda_{k,s}\leq\sup_{\phi\in S_V}\cE_s(\phi,\phi).
		\end{equation}
		Moreover, for all $\phi\in S_V,$ there exists a family of numbers $(c_j)_{j\in\{1,\cdots,k\}}\subset\R$ satisfying $\sum_{j=1}^{k}c^2_j=1$ such that $\phi=\sum_{j=1}^{k}c_je_{j,s}$. From this, we get that
		\begin{align*}
		\cE_s(\phi,\phi)&=\cE_s\Big(\sum_{j=1}^{k}c_je_{j,s},\sum_{j=1}^{k}c_je_{j,s}\Big)=\sum_{i,j=1}^{k}c_ic_j\lambda^{*}_{j,s}\langle e_{i,s},e_{j,s}\rangle_{L^2(\Omega)}\\
		&=\sum_{j=1}^{k}c^2_j\lambda^{*}_{j,s}\leq\max_{j\in\{1,\dots,k\}}\lambda^{*}_{j,s}\sum_{j=1}^{k}c^2_j=\max_{j\in\{1,\dots,k\}}\lambda^{*}_{j,s}.
		\end{align*}
		Hence, from \eqref{uper-bound-of-lambda-k}, we have that
		\begin{equation*}
		\lambda_{k,s}\leq\max_{j\in\{1,\dots,k\}}\lambda^{*}_{j,s}\leq\lambda^{*}_{k,s},
		\end{equation*}
		which therefore implies that
		\begin{equation*}
		\liminf_{\alpha\rightarrow s}\lambda_{k,\alpha}=\lambda^*_{k,s}\geq\lambda_{k,s}.
		\end{equation*}
		Combining both \textbf{\textit{Steps 1}} and \textbf{\textit{2}} we conclude that
		\begin{equation}
		\lim\limits_{\alpha\rightarrow s}\lambda_{k,\alpha}=\lambda_{k,s} 
		\end{equation}
		as wanted. 
	\end{proof}
	Below, we now give the proof of Theorem \ref{differentiability-of-eigenvalues}.
	\begin{proof}[Proof of Theorem \ref{differentiability-of-eigenvalues}]
		By Lemma \ref{right-continuity-of-eigenvalues} and Proposition \ref{higher-sobolev-regularity}, we deduce that the function $u_{s+\sigma}$ is uniformly bounded in $H^s(\Omega)$ with respect to $\sigma$. Therefore after passing to a subsequence, there is $w_s\in H^s(\Omega)$ such that 
		\begin{equation}\label{m1}
		\begin{aligned}
		&u_{s+\sigma}\rightharpoonup w_s\quad\text{weakly in}~~H^s(\Omega),\\
		&u_{s+\sigma}\rightarrow w_s\quad\text{strongly in}~~L^2(\Omega),\\
		&u_{s+\sigma}\rightarrow w_s\quad\text{a.e. in}~~\Omega. 
		\end{aligned}
		\end{equation}
		We wish now to show that $w_s$ is also an eigenfunction corresponding to $\lambda_s$. First of all, from \eqref{m1}, we have in particular that $\|w_s\|_{L^2(\Omega)}=1$ and $\int_{\Omega}w_s\ dx=0$. 
		
		Next, we claim that
		\begin{enumerate}
			\item [$(a)$] $\cE_{s+\sigma}(u_{s+\sigma},\phi)\to \cE_s(w_s,\phi)$
			\item[$(b)$]  $\lambda_{s+\sigma}\int_{\Omega}u_{s+\sigma}\phi\ dx\to \lambda_s\int_{\Omega}w_s\phi\ dx$
		\end{enumerate}
		as $\sigma\to0^+$ for all $\phi\in C^{\infty}_0(\ov\Omega)$. 
		
		We start by proving $(b)$. We write
		\begin{align*}
		\int_{\Omega}(\lambda_{s+\sigma}u_{s+\sigma}-\lambda_sw_s)\phi\ dx=\lambda_s\int_{\Omega}(u_{s+\sigma}-w_s)\phi \ dx+(\lambda_{s+\sigma}-\lambda_s)\int_{\Omega}u_{s+\sigma}\phi\ dx.
		\end{align*}
		From the above decomposition and thanks to  \eqref{m1} and Lemma \ref{right-continuity-of-eigenvalues}, we deduce claim $(b)$.
		
		Regarding $(a)$, we have
		\begin{align}\label{m2}
		\nonumber&|\cE_{s+\sigma}(u_{s+\sigma},\phi)-\cE_s(w_s,\phi)|\\
		\nonumber&\leq\Big|\frac{C_{N,s+\sigma}-C_{N,s}}{2}\Big|\Big|\int_{\Omega}\int_{\Omega}\frac{(w_s(x)-w_s(y))(\phi(x)-\phi(y))}{|x-y|^{N+2s}}\ dxdy\Big|\\
		\nonumber&+\frac{C_{N,s+\sigma}}{2}\Big|\int_{\Omega}\int_{\Omega}\frac{((u_{s+\sigma}(x)-u_{s+\sigma}(y))-(w_s(x)-w_s(y)))(\phi(x)-\phi(y))}{|x-y|^{N+2s}}\ dxdy\Big|\\
		\nonumber&+\frac{C_{N,s+\sigma}}{2}\Big|\int_{\Omega}\int_{\Omega}\Big(|x-y|^{-2\sigma}-1\Big)\frac{(u_{s+\sigma}(x)-u_{s+\sigma}(y))(\phi(x)-\phi(y))}{|x-y|^{N+2s}}\ dxdy\Big|\\
		&=:I_{\sigma}+II_{\sigma}+III_{\sigma}.
		\end{align}
		Since $w_s, \phi\in H^s(\Omega)$ and $s\mapsto C_{N,s}$ is of class $C^1$, then from Cauchy-Schwarz inequality, we get that 
		\begin{equation}\label{m3}
		I_{\sigma}\leq c|C_{N,s+\sigma}-C_{N,s}|\to0\quad\quad\text{as}~~\sigma\to 0^+.
		\end{equation}
		Now, using \eqref{bound-of-constant-c-n-s} and the fact that $u_{s+\sigma}\rightharpoonup w_s$ weakly in $H^s(\Omega)$, one gets
		\begin{equation}\label{m4}
		II_{\sigma}\to0\quad\quad\text{as}~~\sigma\to0^+.
		\end{equation} 
		On the other hand, recalling \eqref{bound-of-constant-c-n-s},  \eqref{n2} and \eqref{n9}, we have
		\begin{align*}
		III_{\sigma}&\leq \sigma c\int_{\Omega}\int_{\Omega}\frac{|u_{s+\sigma}(x)-u_{s+\sigma}(y)||\phi(x)-\phi(y)|}{|x-y|^{N+2s}}|\log|x-y||\ dxdy\\
		&+\sigma c\int_{\Omega}\int_{\Omega}\frac{|u_{s+\sigma}(x)-u_{s+\sigma}(y)||\phi(x)-\phi(y)|}{|x-y|^{N+2s+2\sigma}}|\log|x-y||\ dxdy.
		\end{align*}
		Arguing as in Section \ref{differentiability-for-general-value-s}, one obtains
		\begin{equation}\label{m5}
		III_{\sigma}\leq \sigma c\to0\quad\quad\text{as}~~\sigma\to0^+.
		\end{equation}
		From \eqref{m3}, \eqref{m4} and \eqref{m5}, it follows from \eqref{m2} that
		\begin{equation}
		\cE_{s+\sigma}(u_{s+\sigma},\phi)\to\cE_s(w_s,\phi)\quad\quad\text{as}~~\sigma\to 0^+,
		\end{equation}
		yielding claim $(a)$.
		
		Finally, using that $u_{s+\sigma}$ is solution to
		\begin{equation}
		\cE_{s+\sigma}(u_{s+\sigma},\phi)=\lambda_{s+\sigma}\int_{\Omega}u_{s+\sigma}\phi\ dx
		\end{equation}
		for all $\phi\in C^{\infty}_0(\ov\Omega)$, one deduces from claims $(a)$ and $(b)$ that $w_s$ is solution to
		\begin{equation}\label{m5'}
		\cE_s(w_s,\phi)=\lambda_s\int_{\Omega}w_s\phi\ dx
		\end{equation}
		and from this, one concludes that $w_s$ with $\|w_s\|_{L^2(\Omega)}=1, \int_{\Omega}w_s\ dx=0$ is an eigenfunction corresponding to $\lambda_s$.
		
		Coming back to the proof of \eqref{C-1-regularity-of-eigenvalues}, since  $u_{s+\sigma}\in H^{s+\sigma}(\Omega)\subset H^s(\Omega)$, one can use it as an admissible function in the definition of $\lambda_s$ to get
		\begin{equation}\label{m6}
		\lambda_s\leq\cE_s(u_{s+\sigma},u_{s+\sigma})=\frac{C_{N,s}}{2}\int_{\Omega}\int_{\Omega}\frac{(u_{s+\sigma}(x)-u_{s+\sigma}(y))^2}{|x-y|^{N+2s}}\ dxdy.
		\end{equation}
		Now, from \eqref{m6}, we have 
		\begin{align*}
		&\lambda_{s+\sigma}-\lambda_s=\cE_{s+\sigma}(u_{s+\sigma},u_{s+\sigma})-\lambda_s\\
		&\geq
		\frac{C_{N,s+\sigma}}{2}\int_{\Omega}\int_{\Omega}\frac{(u_{s+\sigma}(x)-u_{s+\sigma}(y))^2}{|x-y|^{N+2(s+\sigma)}}\ dxdy-\frac{C_{N,s}}{2}\int_{\Omega}\int_{\Omega}\frac{(u_{s+\sigma}(x)-u_{s+\sigma}(y))^2}{|x-y|^{N+2s}}\ dxdy\\
		&=\frac{C_{N,s+\sigma}-C_{N,s}}{2}\int_{\Omega}\int_{\Omega}\frac{(u_{s+\sigma}(x)-u_{s+\sigma}(y))^2}{|x-y|^{N+2s+2\sigma}}\ dxdy\\
		&\ \ \ \ \ \ \ \ \ \ \ \ \ \ \ \ \ \ \ \ +\frac{C_{N,s}}{2}\int_{\Omega}\int_{\Omega}(|x-y|^{-2\sigma}-1)\frac{(u_{s+\sigma}(x)-u_{s+\sigma}(y))^2}{|x-y|^{N+2s}}\ dxdy\\
		&=\frac{C_{N,s+\sigma}-C_{N,s}}{C_{N,s+\sigma}}\lambda_{s+\sigma} +\frac{C_{N,s}}{2}\int_{\Omega}\int_{\Omega}(|x-y|^{-2\sigma}-1)\frac{(u_{s+\sigma}(x)-u_{s+\sigma}(y))^2}{|x-y|^{N+2s}}\ dxdy
		\end{align*}
		Hence,
		\begin{align}\label{q}
		\nonumber\liminf_{\sigma\rightarrow0^+}\frac{\lambda_{s+\sigma}-\lambda_s}{\sigma}&\geq\lim\limits_{\sigma\rightarrow0^+}\frac{C_{N,s+\sigma}-C_{N,s}}{\sigma}\frac{\lambda_{s+\sigma}}{C_{N,s+\sigma}}\\
		&+\lim\limits_{\sigma\rightarrow0^+}\frac{C_{N,s+\sigma}}{2}\int_{\Omega}\int_{\Omega}\frac{(u_{s+\sigma}(x)-u_{s+\sigma}(y))^2}{|x-y|^{N+2s}}\frac{|x-y|^{-2\sigma}-1}{\sigma}\ dxdy.
		\end{align}
		Next, from \eqref{n2}
		\begin{align*}
		\frac{|x-y|^{-2\sigma}-1}{\sigma}=\frac{\exp(-2\sigma\log|x-y|)-1}{\sigma}=-2\log|x-y|\int_{0}^{1}\exp(-2\sigma t\log|x-y|)\ dt.
		\end{align*}
		Therefore,
		\begin{equation}\label{mean-value-theorem}
		\frac{|x-y|^{-2\sigma}-1}{\sigma}\rightarrow-2\log|x-y|\quad\text{as}\quad\sigma\rightarrow0^+.
		\end{equation}	
		Using this and recalling \eqref{logarithmic-decays}, we apply Lebesgue's Dominated Convergence Theorem in \eqref{q}, thanks to Proposition \ref{higher-sobolev-regularity}, to get that
		\begin{align}\label{m7}
		\liminf_{\sigma\rightarrow0^+}\frac{\lambda_{s+\sigma}-\lambda_s}{\sigma}\geq\frac{\partial_sC_{N,s}}{C_{N,s}}\lambda_s
		-C_{N,s}\int_{\Omega}\int_{\Omega}\frac{(w_{s}(x)-w_{s}(y))^2}{|x-y|^{N+2s}}\log|x-y|\ dxdy.
		\end{align}
		that is
		\begin{equation}\label{lim-inf-gamma-sigma}
		\liminf_{\sigma\rightarrow0^+}\frac{\lambda_{s+\sigma}-\lambda_s}{\sigma}\geq J_s(w_s)\geq\inf\{J_s(u):u\in \M_s\}.
		\end{equation}
		We now show the reverse inequality i.e.,
		\begin{equation}\label{lim-sup-gamma-sigma-}
		\limsup_{\sigma\rightarrow0^+}\frac{\lambda_{s+\sigma}-\lambda_s}{\sigma}\leq\inf\{J_s(u):u\in \M_s\}.
		\end{equation}
		Thanks to Proposition \ref{higher-sobolev-regularity}, we have that $u_s\in H^{s+\sigma}(\Omega)$ for $\sigma$ sufficiently small. Combining this with $\int_{\Omega}u_s\ dx=0$ and $\|u_s\|_{L^2(\Omega)}=1$, we can use $u_s$ as an admissible function for $\lambda_{s+\sigma}$ to get
		\begin{align*}
		\frac{\lambda_{s+\sigma}-\lambda_s}{\sigma}&\leq\frac{\cE_{s+\sigma}(u_s,u_s)-\cE_s(u_s,u_s)}{\sigma}\\&=\frac{C_{N,s+\sigma}}{2\sigma}\int_{\Omega}\int_{\Omega}\frac{(u_s(x)-u_s(y))^2}{|x-y|^{N+2(s+\sigma)}}\ dxdy-\frac{C_{N,s}}{2\sigma}\int_{\Omega}\int_{\Omega}\frac{(u_s(x)-u_s(y))^2}{|x-y|^{N+2s}}\ dxdy\\&=\frac{C_{N,s+\sigma}-C_{N,s}}{2\sigma}\int_{\Omega}\int_{\Omega}\frac{(u_s(x)-u_s(y))^2}{|x-y|^{N+2s}}\ dxdy\\&\quad\quad+\frac{C_{N,s+\sigma}}{2}\int_{\Omega}\int_{\Omega}\frac{(u_s(x)-u_s(y))^2}{|x-y|^{N+2s}}\frac{|x-y|^{-2\sigma}-1}{\sigma}\ dxdy.
		\end{align*}
		By letting $\sigma\rightarrow0^+$ and 
		applying Lebesgue's Dominated Convergence Theorem and considering once again \eqref{mean-value-theorem} and Proposition \ref{higher-sobolev-regularity}, we have that
		\begin{align*}
		\limsup_{\sigma\rightarrow0^+}\frac{\lambda_{s+\sigma}-\lambda_s}{\sigma}\leq J_s(u_s).
		\end{align*}
		Since the above inequality does not depends on the choice of $u_s\in \M_s$, we have that
		\begin{align}\label{lim-sup-gamma-sigma}
		\limsup_{\sigma\rightarrow0^+}\frac{\lambda_{s+\sigma}-\lambda_s}{\sigma}\leq\inf\{J_s(u):u\in \M_s\}.
		\end{align}
		Putting together \eqref{lim-inf-gamma-sigma} and \eqref{lim-sup-gamma-sigma} we infer that
		\begin{equation}\label{differentiability-for-s-neq-1/2}
		\lim\limits_{\sigma\rightarrow0^+}\frac{\lambda_{s+\sigma}-\lambda_s}{\sigma}=\inf\{J_s(u):u\in \M_s\}\equiv\partial^+_s\lambda_s.
		\end{equation}
		Finally, from Proposition \ref{higher-sobolev-regularity} we easily conclude that the infimum in \eqref{C-1-regularity-of-eigenvalues} is achieved.  
	\end{proof}
	\begin{remark}
		By a similar argument as above, one can also prove that the map $(0,1)\ni s\mapsto \lambda_{1,s}$ is left differentiable. However, due to the non-simplicity of $\lambda_{1,s}$, the right and left derivative $\partial^+_s\lambda_{1,s}$ and $\partial^-_s\lambda_{1,s}$ might not be equal.
	\end{remark}


\begin{thebibliography}{10}
		
		\bibitem{andreu2010nonlocal} F. Andreu-Vaillo, J. M. Maz\'{o}n, J. D. Rossi, and J. J. Toledo-Melero, \emph{Nonlocal diffusion problems.} No. 165. American Mathematical Soc., 2010. 
		
		\bibitem{antil2017spectral} H. Antil and S. Bartels, \emph{Spectral approximation of fractional PDEs in image processing and phase field modeling.} Comput. Methods Appl. Math, 17.4 (2017): 661-678.
		
		
		
		\bibitem{antil2021approximation} H. Antil, S. Bartels, and A. Schikorra, \emph{Approximation of fractional harmonic maps.} arXiv preprint arXiv:2104.10049 (2021).
		
		\bibitem{antil2018optimization} H. Antil, E. Ot\'{a}rola, and A. J. Salgado, \emph{Optimization with respect to order in a fractional diffusion model: analysis, approximation and algorithmic aspects.} 
		J. Sci. Comput, 77.1 (2018): 204-224. 
		
		
		\bibitem{biccari2018poisson} U. Biccari and V. Hern\'{a}ndez-Santamar\'{i}a. \emph{The Poisson equation from non-local to local.} Electron. J. Differential Equations, 2018 (2018).
		
		
		\bibitem{bogdan2003censored} K. Bogdan, K. Burdzy and Z.-Q. Chen, \emph{Censored stable processes.} Probab. Theory Related Fields, 2003, vol. 127, no 1, p. 89-152.
		
		
		
		
		\bibitem{brezis2018gagliardo} H. Brezis and P. Mironescu, \emph{Gagliardo-Nirenberg inequalities and non-inequalities: The full story.} Ann. Inst. H. Poincar\'{e} Anal. Non Lin\'{e}aire. Vol. 35. No. 5. Elsevier Masson, 2018.
		
		
		
		\bibitem{burkovska2020affine} O. Burkovska and M. Gunzburger, \emph{Affine approximation of parametrized kernels and model order reduction for nonlocal and fractional Laplace models.} SIAM J. Numer. Anal. 58.3 (2020): 1469-1494.
		
		
		\bibitem{chen2018dirichlet} H. Chen, \emph{The Dirichlet elliptic problem involving regional fractional Laplacian.} J. Math. Phys.  2018, vol. 59, no 7, p. 071504.
		
		
		\bibitem{chen2020non-existence} H. Chen and Y. Wei, \emph{Non-existence of Poisson problem involving regional
			fractional Laplacian with order in (0, 1/2].} arXiv preprint \href{https://arxiv.org/abs/2007.05775v1}{https://arxiv.org/abs/2007.05775v1} (2020).
		
		\bibitem{chen2019dirichlet} H. Chen and T. Weth, \emph{The Dirichlet problem for the logarithmic Laplacian.} Comm. Partial Differential Equations. 44.11 (2019): 1100-1139.
		
		
		\bibitem{deblassie2005alpha}  R. D.  DeBlassie and  P. J. M{\'e}ndez-Hern{\'a}ndez, \emph{$\alpha$-continuity properties of the symmetric $\alpha$-stable process.} Trans. Amer. Math. Soc. vol. 359, no. 5, pp. 2343-2359, 2007.
		
		
		\bibitem{del2015first} L. M. Del Pezzo and A. M. Salort, \emph{The first non-zero Neumann p-fractional eigenvalue.} Nonlinear Anal. 118 (2015): 130-143.
		
		
		\bibitem{demengel2012functional}
		F. Demengel, G. Demengel and R. Ern\'{e}, \emph{Functional spaces for the theory of elliptic partial differential equations.} London: Springer, 2012.
		
		\bibitem{di2012hitchhiker's} E. Di Nezza, G. Palatucci and E. Valdinoci. \emph{Hitchhiker's guide to the fractional Sobolev spaces.} Bull. Sci. Math. 5.136 (2012): 521-573.
		
		
		\bibitem{fall2020regional} M. M. Fall, \emph{Regional fractional Laplacians: Boundary regularity.} J. Differential Equations. 320 (2022): 598-658.
		
		
		
		\bibitem{feulefack2022small} P. A. Feulefack, S. Jarohs, and T. Weth, \emph{Small order asymptotics of the Dirichlet eigenvalue problem for the fractional Laplacian.} J. Fourier Anal. Appl. 28.2 (2022): 1-44.
		
		
		
		\bibitem{guan2006integration} Q.-Y. Guan, \emph{Integration by parts formula for regional fractional Laplacian.} Comm. Math. Phys. vol. 266, no 2, p. 289-329.
		
		
		\bibitem{guan2006reflected} Q.-Y. Guan and Z.-M. Ma, \emph{Reflected symmetric $\alpha$-stable processes and regional fractional Laplacian.} Probab. Theory Related Fields 134.4 (2006): 649-694.
		
		
		\bibitem{jarohs2020new} S. Jarohs, A. Salda$\tilde{\text{n}}$a and T. Weth, \emph{A new look at the fractional Poisson problem via the Logarithmic Laplacian.} J. Funct. Anal.  (2020): 108732
		
		
		\bibitem{pellacci2018best} B. Pellacci and G. Verzini, \emph{Best dispersal strategies in spatially heterogeneous environments: optimization of the principal eigenvalue for indefinite fractional Neumann problems.} J. Math. Biol. 76.6 (2018): 1357-1386.
		
		
		\bibitem{sprekels2017new} J. Sprekels and E. Valdinoci, \emph{A new type of identification problems: optimizing the fractional order in a nonlocal evolution equation.} SIAM J. Control Optim. 55.1 (2017): 70-93.
		
		
		
		\bibitem{temgoua2021eigenvalue} R. Y. Temgoua and T. Weth, \emph{The eigenvalue problem for the regional fractional Laplacian in the small order limit.} Preprint (2021).
		
		\bibitem{terracini2018s} S. Terracini, G. Tortone and S. Vita, \emph{On s-harmonic functions on cones.} Anal. PDE 11.7 (2018): 1653-1691.
		
		
		\bibitem{warma2015fractional} M. Warma, \emph{The fractional relative capacity and the fractional Laplacian with Neumann and Robin boundary conditions on open sets.} Potential Anal. 2015, vol. 42, no 2, p. 499-547.
		
		
	\end{thebibliography}
	\bibliographystyle{ieeetr}

\end{document}